\documentclass[10.5pt, reqno]{amsart}
\usepackage[margin=3cm]{geometry}

\usepackage{array,booktabs,tabularx}
\usepackage{graphicx}
\usepackage{amssymb,amsmath, amsthm}
\usepackage{enumerate}
\usepackage{xfrac} 
\usepackage{nicefrac}
\usepackage{color}
\usepackage{hyperref}
\usepackage{esint}
\usepackage{mathtools}
\usepackage{dsfont}
\usepackage{ esint }
\usepackage[show]{ed}
\usepackage{comment}
\usepackage{enumitem}
\usepackage{soul}
\usepackage{needspace}
\usepackage[normalem]{ulem}
\usepackage{tikz}
\usepackage{float}

\theoremstyle{plain}

\newtheorem{thm}{Theorem}[section]
\newtheorem{lem}{Lemma}[section]
\newtheorem{prop}{Proposition}[section]
	
\newtheorem{assum}{Assumption}

\theoremstyle{definition}
\newtheorem{defn}{Definition}[section]
\newtheorem{rem}{Remark}[section]

\allowdisplaybreaks[2]

\newenvironment{proof1}[1]{\begin{trivlist} \item[] {\em Proof of #1:}}{\newline \textcolor{white}{.}\hfill $\Box$
                      \end{trivlist}}

\newcommand{\ud}{\,\mathrm{d}}
\newcommand{\la}{\lambda}

\newcommand{\R}{\mathbb{R}}

\newcommand{\pa}{\partial}
\newcommand{\eps}{\epsilon}

\newcommand{\cs}{$\clubsuit$ }

\newcommand{\paSA}{\pa S \backslash A}
\newcommand{\paOA}{\pa \Omega \backslash A}
\newcommand{\norm}[1]{\left\lVert#1\right\rVert}
\newcommand{\dv}{\operatorname{dv}}
\newcommand{\x}{{\bf x}}

\newcommand{\blue}[1]{{\color{blue}{#1}}}

\newcommand{\yai}[1]{{\color{violet}{#1}}}

\title[Quantitative bounds on Impedance-to-Impedance operators]{Quantitative bounds on Impedance-to-Impedance operators with applications to fast direct solvers for PDEs}

\author[T. Beck]{Thomas Beck}
\email{tbeck7@fordham.edu}
\address{Department of Mathematics, Fordham University, John Mulcahy Hall, Bronx, NY 10458}
  
  \author[Y. Canzani]{Yaiza Canzani}
\email{canzani@email.unc.edu}
\address{Department of Mathematics, University of North Carolina at Chapel Hill \\ CB\#3250
  Phillips Hall \\ Chapel Hill, NC 27599}

\author[J.L. Marzuola]{Jeremy L. Marzuola}
\email{marzuola@math.unc.edu}
\address{Department of Mathematics, University of North Carolina at Chapel Hill \\ CB\#3250
  Phillips Hall \\ Chapel Hill, NC 27599}

\date{\today}                                           

\begin{document}
\maketitle

\begin{abstract}

We prove quantitative norm bounds for a family of operators involving impedance boundary conditions on  convex, polygonal domains. A robust numerical construction of Helmholtz scattering solutions in variable media via the Dirichlet-to-Neumann operator involves a decomposition of the domain into a sequence of rectangles of varying scales and constructing impedance-to-impedance boundary operators on each subdomain. Our estimates in particular ensure the invertibility, with quantitative bounds in the frequency, of the merge operators required to reconstruct the original Dirichlet-to-Neumann operator in terms of these impedance-to-impedance operators of the sub-domains. A key step in our proof is to obtain Neumann and Dirichlet boundary trace estimates on solutions of the impedance problem, which are of independent interest. In addition to the variable media setting, we also construct bounds for similar merge operators in the obstacle scattering problem.
   
\end{abstract}

\section{Introduction and statement of results}

 In this article we prove norm bounds for a family of elliptic operators in convex, polygonal domains with impedance boundary conditions. Our framework is motivated by the recent need for studying the effect of impedance boundary conditions in implementations of numerical methods that compute the Dirichlet-to-Neumann map for Schr\"odinger operators on planar domains.  The computation of the Dirichlet-to-Neumann map is an important problem in the study of elliptic boundary value problems, with applications in imaging (\cite{nachman1996global,fang2009viable,wadbro2010high,feldman2019calderon}, \cite{wang20113d}), computation of high energy eigenvalues (\cite{barnett2014fast}), boundary trace estimates (\cite{han2015sharp,barnett2018comparable}), and in scattering theory for solutions to the Helmholtz equation (\cite{gillman2015spectrally,pedneault2017schur,graham2019helmholtz}).  


 Constructing numerical methods that solve elliptic problems with impedance boundary conditions on planar domains is a fundamental problem in numerical analysis. For instance, they arise when computing single frequency scattering solutions of the Helmholtz equation.  There are a large number of methods that treat the latter problem, see for instance the classical results   \cite{descloux1983accurate,kirsch1990convergence,kirsch1994analysis,melenk1995generalized,monk1999least}, as well as the more recent results \cite{barnett2010exponentially,baskin2016sharp,barnett2018comparable,pedneault2017schur,fortunato2020ultraspherical,mascotto2020fem}.  The recent book \cite{martinsson2019fast} contains a very thorough overview of the subject.  The approach to solving elliptic boundary value and scattering problems using Impedance-to-Impedance operators has become known as the hierarchical Poincare-Steklov method.

{One robust approach for numerically constructing scattering solutions via the Dirichlet-to-Neumann (DtN) map is a divide and conquer approach, which involves domain decomposition through splitting the domain into smaller more manageable sub-domains.  One of the first results that used impedance boundary conditions in these methods was \cite{benamou1997domain}.  Impedance boundary conditions arise both in classical domain domain decomposition methods, where an elliptic boundary problem is solved as a standard linear system, as well as in the building of direct solvers as in \cite{gillman2015spectrally,pedneault2017schur}, where an approximate inverse is constructed.} The key feature of these numerical constructions of the DtN operator is the decomposition of the large rectangular domain containing the scatterers into small polygonal components.  Computing the DtN operator of each small component is ill-conditioned numerically, due to the possibility of hitting a resonance of the DtN operator. Therefore, one instead imposes impedance boundary conditions on each sub-domain and computes the resulting impedance-to-impedance operators (see Definition \ref{defn:R-Omega} for the precise definition of these operators).  Working with these operators removes the possibility of artificially introducing a resonance to the problem. 
 These domains are then merged back together through a canonical process to recover the impedance-to-impedance and DtN operators of the original domain. We describe this process in Section \ref{sec:ItIapps}.


A version of the divide and conquer approach described above referred to as the hierarchical Poincare-Steklov method was implemented very successfully in the work of Gillman-Barnett-Martinsson \cite{gillman2015spectrally}, where the authors study a single frequency Helmholtz scattering problem in the presence of an inhomogeneous medium.  Using a spectral discretization on a Chebyshev grid, one can thus solve the elliptic problem on each sub-domain to high accuracy.  The merge process is done on a sequence of rectangles of varying scales in order to reconstruct the exterior DtN operator  (see \cite[Figure 2]{gillman2015spectrally}). A similar merge procedure using domain decomposition with impedance boundary conditions has also been applied in obstacle scattering problems, with multiple disjoint scatterers, see \cite{pedneault2017schur}.  A similar approach was developed in \cite{gillman2014direct} using DtN maps instead of Impedance-to-Impedance maps.


The procedure given in \cite{gillman2015spectrally} requires the invertibility of a merge operator, derived from equations $(2.15)$ and $(2.16)$ of Section $2.4$ of that paper, coming from impedance-to-impedance operators on adjacent domains. The invertibility of this operator is assumed in \cite{gillman2015spectrally}, and holds for all of their numerical computations. As a consequence of our main theorems, we prove that this operator is indeed invertible and establish frequency dependent estimates on its inverse.


Impedance boundary conditions and impedance-to-impedance operators have also been used in many other numerical schemes. For example, when computing scattering solutions of the Helmholtz equation in variable media, an important model problem is to replace the outgoing Sommerfeld radiation condition satisfied by the scattering field by an impedance boundary condition.  {In this case, impedance boundary conditions are used as first order absorbing boundary conditions and were proposed in \cite{engquist1977absorbing_moc,engquist1977absorbing,engquist1979radiation,bayliss1982boundary}.  Estimates on such models have been derived recently in \cite{graham2019helmholtz}, as well as in \cite{galkowski2021local} in the high frequency limit.} The key feature of this model is that this boundary condition is imposed on the boundary of a large rectangular domain outside of which the wave speed is constant, and allows for a numerical study of the problem. Impedance-to-impedance operators have also been used in \cite{joly2006} when calculating periodic wave-guides with perturbations and in \cite{fliss2015} for computing modes in photonic crystal wave-guides.

\subsection{Set-up of the problem and statement of results}
\label{sec:ItIapps}

We study the impedance-to-impedance (ItI) operator for a domain $\Omega \subset \R^2$, defined as follows. Given $f\in L^2(\pa\Omega)$, consider the problem
\begin{equation}
\label{eqn:gbm}
\begin{cases}
    \Delta u + k^2 V u  = 0 &\text{ in } \Omega, \\
    \pa_{\nu}u + iku  = f &\text{ on } \pa\Omega,
\end{cases}
\end{equation}
where $\Delta$ is the negative definite Laplacian, $k>0$, and $V\in L^\infty(\Omega)$ is a positive scattering potential.
Then, under suitable assumptions on the domain $\Omega$ and potential $V$, problem \eqref{eqn:gbm} has a unique solution $u\in H^1(\Omega)$ such that $\pa_{\nu}u - iku \big|_{\pa\Omega}\in L^2(\pa\Omega)$. Here $\nu$ is a unit  {normal vector} to $\pa\Omega$. The  ItI operator $R_{\Omega}$ is then defined by 
$$
R_{\Omega}f =  \big( \pa_{\nu}u - iku \big)\Big|_{\pa\Omega}.
$$
An important feature of the ItI operator is that, unlike the DtN operator, it is a unitary transformation and hence numerically very well-conditioned.  Our main results concern the case where $\Omega$ is the rectangle $[0,2]\times[0,1]$, or the unit squares which are the two halves of this rectangle. We begin with the following set-up.

Let $S = [0,1]\times[0,1]$ be the unit square, and $V_1$, $V_2\in L^{\infty}(S)$ be non-negative potentials.  Writing the right hand edge as $A = \{1\}\times[0,1]$, and given $f\in L^2(A)$, $k>0$, suppose that for $j=1,2$, $u_j$ solve the problem
\begin{equation}\label{eqn:u-i1}
\begin{cases} 
(\Delta + k^2V_j)u_j  = 0 & \text{ in } S, \\ 
\pa_{\nu} u_j + iku_j  = 0 &\text{ on } \pa S\backslash A, \\ 
\pa_{\nu} u_j + iku_j  = f &\text{ on } A.
\end{cases}
\end{equation}
Here $\nu$ is the outward pointing unit normal to $\pa S$. We view the potentials $V_1$, $V_2$ as coming from the smooth potential on the rectangle $[0,2]\times[0,1]$, in the following sense.


\begin{assum}[$V_j$ are non-trapping with respect to $(1,0)$] \label{assum:potential-main}
There exist $V \in C^1([0,2]\times[0,1]) $  and  $c>0$ such that the potentials $V_1$, $V_2\in C^{1}(S)$ defined by
\begin{equation*}
V_1(x,y) = V(x,y), \qquad V_2(x,y) = V(2-x,y),
\end{equation*}
satisfy 
\begin{align} \label{eqn:constants}
2V_j(x,y)+  (x-1,y) \!\cdot\! \nabla V_j(x,y)  \geq c 
\end{align}
for all $(x,y)\in S$, and $j = 1,2$.
\end{assum}

 In particular, the condition in \eqref{eqn:constants} ensures that $V_j(x,y)\geq\tfrac{1}{2}c$ for all $(x,y)\in S$. Assumption \ref{assum:potential-main} implies that both $V_1$ and $V_2$ are non-trapping with respect to the vertex $(1,0)$ in a way which we will make precise in Remark \ref{rem:non-trapping} below. The problem \eqref{eqn:u-i1} has a unique solution $u_j\in H^1(S)$ (see Proposition 2.1 in \cite{gillman2015spectrally}). Under Assumption \ref{assum:potential-main}, this solution $u_j$ satisfies $\pa_{\nu}u_j\big|_{A}\in L^2(A)$ and an elliptic estimate giving $L^2$-norm bounds on the boundary data (see Proposition~\ref{prop:elliptic1} below). Therefore, we may now define the operators $R_j$ on $L^2(A)$:
\begin{defn} \label{defn:R-Omega}
{Let $V_1,V_2$ satisfy  Assumption \ref{assum:potential-main}, and let $k>0$. For $j=1,2$  we define the impedance-to-impedance operator $R_j$ on  $ L^2(A)$ by
\begin{align}\label{eqn:R-Omega}
R_j f = \big( \pa_{\nu}u_j - iku_j \big)\Big|_{A},
\end{align}
where $u_j$ is as in {\eqref{eqn:u-i1}}.}
\end{defn}

By Assumption \ref{assum:potential-main} on the potentials $V_j$, the operators $R_j$ satisfy boundedness properties involving the following  modified versions of the Sobolev $H^1$-norm, adapted to the size of the $k$ parameter.
For functions $h\in L^2(A)$  {and $k>0$} we define the norm
\begin{align}\label{defn:H1-A}
 \norm{h}_{\mathcal{H}^1_k(A)} & := \norm{kh}_{L^2(A)} + \norm{\pa_{\tau}h}_{L^2(A)}.
\end{align}
Here $\pa_{\tau} = \pa_{y}$ is the tangential derivative on $A$. 
{In addition, we let $\mathcal{H}^1_k(A)\subset L^2(A)$ be the Sobolev space defined using the norm  $ \norm{\cdot}_{\mathcal{H}^1_k(A)}$.}

Our main theorems concern the invertibility of the operator $I-R_1R_2$ on appropriately chosen function spaces. When this inverse exists, we denote it by
\begin{equation}
    \label{merge}
    W = (I-R_1R_2)^{-1}.
\end{equation} 



Before stating our estimates on $W$, we first explain how $W$ appears when constructing the ItI operator, $R_\Omega$, on $\Omega = [0,2]\times[0,1]$ by a merge procedure involving ItI operators on the two squares. Denoting $S_1$ and $S_2$ to be the two square halves of $\Omega$, we consider the two elliptic problems
\begin{equation*}
\begin{cases} 
(\Delta + k^2 V)v_j  = 0 & \text{ in } S_j, \\ 
\pa_{\nu} v_j + ikv_j  = 0 &\text{ on } \pa S_j\backslash A, \\ 
\pa_{\nu} v_j + ikv_j  = f &\text{ on } A
\end{cases}
\qquad \qquad \qquad
\begin{cases} 
(\Delta + k^2 V)w_j  = 0 & \text{ in } S_j, \\ 
\pa_{\nu} w_j + ikw_j  = g &\text{ on } \pa S_j\backslash A, \\ 
\pa_{\nu} w_j + ikw_j  = 0 &\text{ on } A.
\end{cases}
\end{equation*}
for $f \in L^2 (A)$ and $g \in L^2 (\pa S_j \backslash A)$.
We can then define the ItI operators $R_j$ and $Q_j$ on $L^2(A)$ and $L^2(\pa S_j\backslash A)$ respectively 
by
\begin{align*}
    R_jf = \pa_{\nu}v_j - ikv_j\big|_A, \qquad Q_jg = \pa_{\nu}w_j - ikw_j\big|_{A}.
\end{align*}
An illustration of this set-up is given in a simple setting in Figure \ref{fig:merge1}.

\begin{figure}[h]
		\includegraphics[width=6cm]{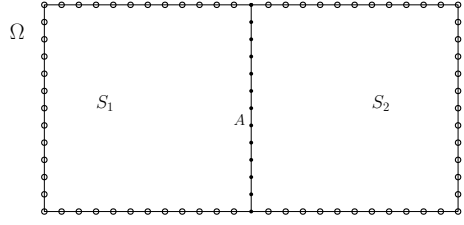}
	\caption{ An example rectangle $\Omega$ formed by regions $S_1$ and $S_2$ merged together on edge $A$. We distinguish between $\partial \Omega \setminus A$ and $A$ by using circles and dots respectively.}
	\label{fig:merge1}
\end{figure}

The operators $R_j$ and $Q_j$ are given in terms of the ItI operators $R_{S_j}$ of the two unit squares. Note that the operators $R_j$ are the same as those in Definition \ref{defn:R-Omega}, which can be seen by transforming the square $S_2$ onto $S = [0,1]\times[0,1]$ via the transformation $(x,y)\mapsto (2-x,y)$. Then, to construct the ItI operator, $R_{\Omega}$, of the rectangle $\Omega$ in terms of the operators $R_j$, $Q_j$, there is the following merge procedure: Given $h\in L^2(\pa\Omega)$, we define $u\in H^1(\Omega)$ to be the solution of
\begin{equation*}
\begin{cases} 
(\Delta + k^2V )u  = 0 & \text{ in } \Omega, \\ 
\pa_{\nu} u + iku  = h &\text{ on } \pa \Omega,
\end{cases}
\end{equation*}
so that $R_{\Omega}h = \pa_{\nu}u -iku\big|_{\pa\Omega}$. To recover $R_{\Omega}h$, we set $f_j = \pa_{\nu_j}u + iku\big|_A$ and $g_j = \pa_{\nu_j}u - iku\big|_A$ to be \textit{incoming} and \textit{outgoing} data on the boundary $A$ shared by the two squares $S_j$. Here $\nu_j$ is the outward pointing unit normal to $S_j$ on $A$, so that $\pa_{\nu_1} = - \pa_{\nu_2} = \pa_x$, and in particular $f_1 = -g_2$, $f_2 = -g_1$. Then, to construct $R_{\Omega}$ in terms of $R_j$ and $Q_j$, it is necessary to express the data $f_j$ on $A$ in terms of $h$. Writing $h_j = h\big|_{\pa\Omega\cap\pa S_j}$ gives the system of equations
\begin{align} \label{merge-eqn1}
\begin{split}
Q_1h_1 + R_1f_1 & = \pa_{\nu_1}u - iku\big|_{A} = g_1 \\
Q_2h_2 + R_2f_2 & = \pa_{\nu_2}u - iku\big|_{A} = g_2.
\end{split}
\end{align}
Using $f_1 = -g_2$, $f_2 = -g_1$, we can rewrite the system in \eqref{merge-eqn1} as
\begin{align} \label{merge-eqn2}
\begin{split}
Q_1h_1 + R_1f_1 & = - f_2 \\
Q_2h_2 + R_2f_2 & = -f_1.
\end{split}
\end{align}
This system is invertible, in that one can write $f_1$ and $f_2$ in terms of $h_1$ and $h_2$, precisely when the operator $I-R_1R_2$ is invertible. In this case, setting $W = (I-R_1R_2)^{-1}$ as above, the system of equations in \eqref{merge-eqn2} can be inverted to write $f_2$ as
\begin{align} \label{merge-eqn3}
    f_2 = -WQ_1h_1 + WR_1Q_2h_2.
\end{align}
The function $f_1$ can then be recovered  from the second equation in \eqref{merge-eqn2}. Now that $f_1$ and $f_2$ are prescribed in terms of $h$, we can immediately use $R_{S_1}$ applied to $h|_{\pa\Omega\cap\pa S_1}$ and $f_1$ on $\pa S_1$ and $R_{S_2}$ applied to $h|_{\pa\Omega\cap\pa S_2}$ and $f_2$ on $\pa S_2$ in order to define $R_{\Omega}h$. This procedure  is precisely the merge procedure given in \cite{gillman2015spectrally}, where the invertibility of $I-R_1R_2$ is assumed. Our main theorems give the existence of $W = (I-R_1R_2)^{-1}$, with quantitative bounds in the frequency $k$, which then allows for the full recovery of the $R_\Omega$ map. In addition, in Section \ref{sec:example} we  show that our results also give bounds on the operators $WQ_1$ and $WR_1Q_2$ appearing in \eqref{merge-eqn3}.  {We hope that our main theorems shed some more light into why impedance boundary conditions work so well in classical domain decompositions in numerical methods.}

We now state our main results on the invertibility of $I-R_1R_2$.
\begin{thm} \label{thm:injective}
{Let $V_1,V_2$ satisfy  Assumption \ref{assum:potential-main}, and let $k>0$. Then,  for $R_1, R_2$ as defined in \eqref{eqn:R-Omega}, the operator
$$I-R_1R_2: L^2(A) \to \mathcal{H}^1_k(A)$$
is a bijection, and the inverse operator, $W$ defined in \eqref{merge},  is bounded.}
\end{thm}

In particular, this theorem allows for $R_{\Omega}$ to be recovered from the ItI operators of $S_1$ and $S_2$. Next, we obtain explicit estimates on the operator norm of $W$ under control on the difference $\norm{V_2-V_1}_{L^{\infty}(S)}$.

\begin{thm} \label{thm:Main-W}
{Let $V_1, V_2$ satisfy  Assumption \ref{assum:potential-main} and let  $\delta>0$.
 Then, there exist  $\eps_\delta>0$ and $C_{\delta}>0$ such that the following holds. 
  Given  $k>0$, suppose that 
\begin{equation} \label{assum:difference}
\norm{V_2-V_1}_{L^{\infty}(S)} \leq \eps_\delta(1+k)^{-3(1+\delta)}.
\end{equation}
Then,  for $R_1, R_2$ as defined in \eqref{eqn:R-Omega}, the operator $W$ in \eqref{merge} satisfies
\begin{align} \label{eqn:W norm}
\norm{Wg}_{L^2(A)} \leq C_{\delta}\,(1+k)^{3(1+\delta)}\left(\norm{g}_{L^2(A)}  + \norm{\tfrac{1}{k}\pa_\tau g}_{L^2(A)} \right)
\end{align}
for all $g \in \mathcal{H}^1_k(A)$.
In addition, the constant $C_\delta$ depends only on $\delta$, the   constant in \eqref{eqn:constants}, and the $C^1(S)$-norms of $V_1$, $V_2$.
}
\end{thm}

In the theorems above, and throughout this work,  ${\pa_\tau} g$ denotes the tangential derivative of $g$.  We note that the bound \eqref{eqn:W norm} is equivalent to stating that $W = (I-R_1R_2)^{-1}: \mathcal{H}^1_k(A) \to L^2(A)$ is bounded with operator norm
\begin{align*}
\|W\| \leq C_{\delta}\,k^{-1}(1+k)^{3(1+\delta)}.
\end{align*}
We emphasize here that our estimate on $W$ is explicit in the frequency $k$. Using \eqref{assum:difference}, the estimate on $W$ in Theorem \ref{thm:Main-W} will follow from estimates on the operators $I\pm R$. We will prove these estimates on $I\pm R$ in a more general setting, for convex polygons and a potential with a non-trapping assumption (see Theorem \ref{thm:Main} below).   In Section \ref{sec:example1}, we compute explicit solutions to the impedance problem on the square (with a constant potential) in order to discuss the sharpness of the spaces and frequency bounds in our theorems. In particular, we show that $W$ is not bounded as an operator from $L^2(A)$ to itself. We also show that the constant $(1+k)^{3(1+\delta)}$ appearing in Theorem \ref{thm:Main-W} cannot be replaced by $(1+k)^{\alpha}$ for any power $\alpha<\tfrac{1}{2}$. The factor of $(1+k)^{3(1+\delta)}$ appears when bounding the Dirichlet data of $u_j$ on $A$ in terms of its Neumann data on $A$, where $u_j$ satisfies \eqref{eqn:u-i1} (see Proposition \ref{prop:weak2}). In Remark \ref{rem:sharpness} below we pinpoint the estimates in the proof where an improved dependence on $k$ would lead to a smaller power of $1+k$ in the estimate in \eqref{eqn:W norm}.

We prove a very similar statement to Theorem \ref{thm:Main-W} in Section \ref{sec:obs} for the setting of scattering obstacles inside the regions $S_1$ and $S_2$ instead of potentials. Obstacle scattering is widely used in applications, and see, for example, \cite{pedneault2017schur}, where a similar merge procedure to the above is used in this setting. For the obstacle scattering problem, it is not possible to obtain bounds on the impedance problem that are uniform in small $k$ (see Theorem A.6 in \cite{graham2019helmholtz}). For $k$ bounded below, in Section \ref{sec:obs} we obtain quantitative (non-sharp) estimates in terms of $k$ on the analogous merge operators in this setting.


A key step in the proof of these theorems is to obtain a lower bound on the Neumann and Dirichlet traces on $A$ of the solutions, $u_j$, to the impedance problem. {We carry this out for a class of convex polygons in the proofs of Propositions \ref{prop:weak} and \ref{prop:weak2}. The bounds are obtained using appropriately chosen vector fields, adapted to the polygon, and Green's function estimates. These estimates are new and of independent interest, as} estimating the Dirichlet and Neumann boundary data of Laplace eigenfunctions (for example,  \cite{barnett2018comparable}, \cite{christianson2017triangle}) and solutions of the wave equation (for example, \cite{bardos1992sharp}, \cite{tataru-trace}) are well studied questions. We in particular highlight \cite{barnett2018comparable}, where these estimates are used for another numerical purpose, namely to show how to obtain tight inclusion bounds for Dirichlet and Neumann eigenvalues.

\subsection{Applications to the \cite{gillman2015spectrally} merge procedure} \label{rem:rescaling}


Recall that the merge process described above in \eqref{merge-eqn1}, \eqref{merge-eqn2}, and \eqref{merge-eqn3} is the procedure given in \cite{gillman2015spectrally}. There, they propose a numerical method for constructing the Dirichlet-to-Neumann operator of the problem
$$
\begin{cases}
    \Delta u + k^2(1-b)u  = 0 &\text{ in } \Omega, \\
 u  = h &\text{ on } \pa\Omega,
\end{cases}
$$
where $b$ specifies the deviation of the wave speed from the wave speed of a constant background and $V = 1-b$ plays the role of the non-trapping potential above.
This problem has a unique solution $u \in H^1(\Omega)$ for each $h\in H^1(\pa\Omega)$ and for $k$ away from a discrete set of resonant wavenumbers. The  method proposed in \cite{gillman2015spectrally} involves  solving this Dirichlet problem when $\Omega$ is a square or rectangular planar domain, in terms of its ItI operator $R_{\Omega}$. To construct $R_{\Omega}$, they partition $\Omega$ into a hierarchical tree of square or rectangular boxes {$\{\Omega^{\tau}\}_\tau$} (see \cite[Figure 2]{gillman2015spectrally}), and spectrally approximate the ItI operator on each leaf box. To approximate $R_{\Omega}$ they then use the above merge process. More precisely, if $\Omega^{\alpha}$ and $\Omega^{\beta}$ are squares or rectangles with one common side such that $\Omega^{\tau} = \Omega^{\alpha}\cup\Omega^{\beta}$, then $R_{\Omega^{\tau}}$ is constructed in terms of $R_{\Omega^{\alpha}}$ and $R_{\Omega^{\beta}}$ via a merge operator $W^{\alpha\beta}$. 

{In Section \ref{sec:example2}, we show an important application of our main theorems to this merge procedure. Namely, we will show that estimate \eqref{eqn:W norm} implies that the operators $WQ_1$ and $WR_1Q_2$ appearing in \eqref{merge-eqn3} are bounded from $L^2(A)$ to itself, with operator norm uniformly bounded for small $k$. This is crucial in order to allow the merge process to be applied iteratively.} 

For simplicity, we consider the case of merging two squares, but our methods allow for merging rectangles works in exactly the same way. To instead merge two squares $\Omega^{\alpha}$, $\Omega^{\beta}$ of side length $r<1$ (instead of $r=1$), we can first rescale the problem to unit scale. This has the effect of changing $k$ to $rk$ and the potentials $V_j(x,y)$ to $V_j(rx,ry)$ for which Assumptions \ref{assum:potential-main} and \eqref{assum:difference} continue to hold. Since the constant in Theorem \ref{thm:Main-W} is uniform in small $k$, the estimate in the theorem remains valid for the merge operator $W^{\alpha\beta}$ of all smaller squares. This is crucial because, for example, in \cite{gillman2015spectrally}, they use these merge operators for their hierarchical tree of rectangular boxes of decreasing size. Our theorems therefore provide estimates that guarantee the existence and boundedness of $W^{\alpha\beta}$ on appropriate spaces.  In \cite{gillman2015spectrally}, the invertibility and boundedness of $W^{\alpha\beta}$ is simply assumed, whereas here we are able to provide quantitative bounds on the operators of the type $W^{\alpha\beta}$.  Boundedness of a similar merge operator for an obstacle scattering problem was established in \cite{pedneault2017schur}, but without quantitative control, as we provide in Section \ref{sec:obs}.



\subsection*{Outline of the paper} For the rest of the paper we proceed as follows. In Section \ref{sec:main}, we discuss the proofs of the two theorems and record some of the elliptic estimates satisfied by the impedance-to-impedance operators $R_j$. In particular, we show that the proofs of the theorems reduce to estimates on the boundary behavior of solutions of the impedance problem on the square. In Section \ref{sec:proof}, we establish this boundary behavior for the impedance problem for a class of convex polygons. This involves a lower bound on the Dirichlet traces for solutions of a related Neumann problem. We consider some quantitative bounds that follow from our analysis for the case of scattering by a convex obstacle in Section \ref{sec:obs}, and to do this, we require some microlocal analytic bounds that we discuss in Appendix \ref{app:obscont}.  In Section \ref{sec:example}, we discuss the sharpness of the estimates in Theorems \ref{thm:injective} and \ref{thm:Main-W}, and provide a further discussion of how the estimates relate to the use of the operator $W$ in the numerical scheme in \cite{gillman2015spectrally}. 

\subsection*{Acknowledgements} TB was supported by NSF Grant DMS-2042654. YC was supported by the Alfred P. Sloan Foundation and NSF Grant DMS-1900519.  JLM was supported in part by NSF CAREER Grant DMS-1352353 and NSF Applied Math Grant DMS-1909035.  JLM also thanks MSRI for hosting him during the outset of this research project. The authors would like to thank Alex Barnett for introducing us to the problem and giving valuable feedback on a draft of the result.  We also thank Dean Baskin, Adrianna Gillman and Euan Spence for helpful discussions on numerical methods for the Helmholtz equation and comments on early versions of the results.  We also thank Nicolas Burq for insightful conversations about his work on boundary control theory that allowed us to extend to the obstacle case.
\section{Discussion of the estimates for the merge operator $W$} \label{sec:main}
In this section we show how Theorems \ref{thm:injective} and \ref{thm:Main-W} reduce to estimates on the boundary behavior of solutions of the impedance problem. Before considering the merge operator $W$ we first record estimates for the operators $R_j$ from Definition \ref{defn:R-Omega}. 
\begin{prop} \label{prop:elliptic1}
{Let $V \in C^\infty(S)$ satisfy  \eqref{eqn:constants} in Assumption \ref{assum:potential-main}, and let $k>0$.
Then, for $G\in L^2(S)$, $g\in L^2(\pa S)$, the problem }
\begin{equation} \label{eqn:u-i}
\begin{cases}
(\Delta + k^2V)u  = G & \text{ in } S \\
\pa_{\nu} u + iku  = g &\text{ on } \pa S  
\end{cases}
\end{equation}
 has a unique solution $u\in H^1(S)$. Moreover, there exists a constant $C$ (independent of $k$) such that
\begin{align*}
\norm{u}_{\mathcal{H}^1_k(S)}  + \norm{ku}_{L^2(\pa S)} + \norm{\pa_\nu u}_{L^2(\pa S)} + \norm{\pa_\tau u}_{L^2(\pa S)} \leq C\left(\norm{G}_{L^2(S)} + \norm{g}_{{L^2(\partial S)}}\right).
\end{align*}
\end{prop}
In the above proposition, we use the norm 
\begin{align*}
\norm{v}_{\mathcal{H}^1_k(S)} & = \norm{kv}_{L^2(S)} + \norm{\nabla v}_{L^2(S)},
\end{align*}
for functions $v\in H^1(S)$. We also let $\mathcal{H}^1_k(S)\subset L^2(S)$ be the Sobolev space defined using the norm  $ \norm{\cdot}_{\mathcal{H}^1_k(S)}$.

\begin{proof}
In the case of constant potential $V$, 
the control on $\norm{u}_{\mathcal{H}_k^1(S)}$ is given in \cite[Proposition 8.1.4]{melenk1995generalized}. Also, in this case, the control on $u$ and $\nabla u$ on $\pa S$ follow from equation (8.1.5) and the displayed equation above (8.1.7) in \cite{melenk1995generalized}.

The method of proof from \cite{melenk1995generalized} still applies in the case of non-constant potential, provided Assumption \ref{assum:potential-main} holds.
{To see this, as in \cite{melenk1995generalized}, one uses the bilinear form
$$B(u,v)=\int_{S} \nabla u \cdot \nabla \overline{v} \,\dv -k^2 \int_S V \,u \,\overline{v} \dv +ik \int_{\partial S} u \, \overline{v} \, \ud s,$$
where we write ${\bf x}=(x,y)$, $\dv$ for the Euclidean measure on $S$, and $\ud s$ for the induced measure on $\partial S$.
The estimates on $u$ are obtained by using the test function $v({\bf x})= {\bf z}\cdot \nabla u({\bf x})$, and an integration by parts argument, where $\textbf{z}$ is in the interior of $S$. Here we choose ${\bf z} = (x-1+\eps, y-\eps)\in S$, with $\eps>0$ a small constant chosen to ensure that Assumption \ref{assum:potential-main} guarantees that
\begin{align*}
    2V(x,y) + (x-1+\eps,y-\eps)\cdot \nabla V(x,y) \geq \tfrac{1}{2}c.
\end{align*}
Then, the only change in the proof when $V$ is non-constant is  that when integrating by parts one needs to differentiate the potential. Indeed, one obtains
\begin{align}\label{eqn:B}
\text{Re} B(u,v)
&=k^2\int_{S} V |u|^2 \,\dv+ \frac{k^2}{2}\int_{S} ({\bf z}\cdot \nabla V )|u|^2 \,\dv +\frac{1}{2} \int_{\partial S} |\nabla u|^2({\bf z}\cdot \nu) \ud s \notag\\
&  -\frac{k^2}{2} \int_{\partial S} V|u|^2({\bf z}\cdot \nu) \ud s+\text{Re}\, ik \int_{\partial S} u ({\bf z}\cdot \nabla \overline{u}) \ud s .
\end{align}
To obtain the same bound as in the estimate displayed right before \cite[(8.1.7)]{melenk1995generalized} one now uses that Assumption \ref{assum:potential-main} yields
$$
k^2\int_{S} V |u|^2 \,\dv+ \frac{k^2}{2}\int_{S} ({\bf z}\cdot \nabla V )|u|^2 \,\dv \geq \tfrac{1}{2}c  k^2\int_{S}  |u|^2 \,\dv,
$$
and that ${\bf z}\cdot \nu \geq \eps >0$ on $\pa S$.
The rest of the proof remains unchanged.
}
\end{proof}

\begin{rem} \label{rem:non-trapping} The estimate on $V$ coming from \eqref{eqn:constants} in Assumption \ref{assum:potential-main} is used in a crucial way in the above proof. This assumption is natural because it guarantees that $V$ is non-trapping in the following sense. Let $q_0 = (1,0) \in S$ and $r>0$ be such that $B(q_0, r)\subset S$. Then, there exists a time $t_r>0$ such that any trajectory $q(t)$ with $q(0)=q_0$ and corresponding to a bicharacteristic $(q(t), \xi(t)) \in T^*\R^2$ associated to the Hamiltonian flow induced by $H(q, \xi)=|\xi|^2-V(q)$  will leave the ball $B(q_0, r)$ for all time $t>t_r$. To see this we note that 
$$\frac{d^2}{dt^2}|q(t)-q_0|^2 =  4\big(2V(q(t))+   (q(t)-q_0)\cdot \nabla V (q(t))\big).$$
 Indeed, $\frac{d}{dt}|q(t)-q_0|^2=\{ H, |q-q_0|^2  \}=4 (q(t)-q_0)\cdot \xi(t)$, and so
$\frac{d^2}{dt^2}|q(t)-q_0|^2=\{ H, 4(q-q_0)\cdot\xi  \}=4(2|\xi|^2+(q-q_0) \cdot \nabla V)=4\Big(2V(q(t))+(q(t)-q_0) \cdot \nabla V(q(t))\Big).$
It then follows that Assumption \ref{assum:potential-main} yields  $\frac{d}{dt}|q(t)-q(0)|^2 \geq c$  for all time $t$. Therefore, $|q(t)-q(0)|^2\geq c t^2 - t|\dot q(0)|^2$, and so there exists $t_r>0$ such that $|q(t)-q_0|> r$ for all $t>t_r$.
\end{rem}


{We note that for} $k$ bounded away from $0$, the estimates in Proposition \ref{prop:elliptic1} are also contained in Theorem A.6 in \cite{graham2019helmholtz}, where they obtain estimates on the solution to the impedance problem under Assumption \ref{assum:potential-main} on $V_j$ and a class of variable coefficient operators. 

\begin{lem} \label{lem:R-boundedness}
{Let $V_1, V_2$ satisfy  Assumption \ref{assum:potential-main}, and let $k>0$.} 
Then, there exists {$C>0$}, depending only on the constant  in \eqref{eqn:constants}, such that {for $j=1,2,$  and $f \in L^2(A)$},
\begin{align*}
\norm{R_jf}_{L^2(A)} \leq C\norm{f}_{L^2(A)}, \qquad 
\qquad \norm{(I-R_j)f}_{\mathcal{H}^1_k(A)} \leq C\norm{kf}_{L^2(A)}.
\end{align*}
\end{lem}

\begin{proof} 
Note that $R_jf= \big( \pa_{\nu}u_j - iku_j \big)\big|_{A}$ and $(I-R_j) f=2i k u_j \big|_A$. Therefore, the result follows from applying  Proposition \ref{prop:elliptic1} with $G=0$ in $S$ and $g = f$ on $A$ and $g = 0$ on $S\backslash A$.
\end{proof}

To prove Theorems \ref{thm:injective} and \ref{thm:Main-W}, {for $f \in L^2(A)$} we obtain a lower bound on $g = (I-R_1R_2)f$ by writing
\begin{equation}
\label{eqn:decomposition}
(I-R_1 R_2)f = (I-R_1)(I+R_2)f - (R_2-R_1)f.
\end{equation}
Together with a straightforward upper bound on $R_2-R_1$, the key ingredient in the proof is a lower bound on $I\pm R_j$. We will prove the following.
\begin{prop} \label{prop:Main}
{Let $V_1, V_2$ satisfy  Assumption \ref{assum:potential-main}, and let $k>0$.} 
Given $\delta>0$, there exist constants $c^*>0$, $c^*_{\delta}>0$, depending  only on $\delta$ and  the  constant in \eqref{eqn:constants}, such that for all {$j=1,2,$ and } $f\in L^2(A)$,
\begin{align*}
\norm{(I-R_j)f}_{\mathcal{H}^1_k(A)} & \geq c^*\norm{kf}_{L^2(A)}, \\
\norm{(I+R_j)f}_{L^2(A)} & \geq c_{\delta}^*(1+k)^{-3(1+\delta)}\norm{f}_{L^2(A)} .
\end{align*}
In addition, the image of $(I+R_j)$ on $L^2(A)$ is $\mathcal{H}^1_k(A)$, and the image of $(I-R_j)$ on $L^2(A)$ is $L^2(A)$.
\end{prop}
Since $(I-R_j)f = 2iku_j\big|_{A}$ and $(I+R_j)f = 2\pa_{\nu}u_j\big|_{A}$, Proposition \ref{prop:Main} provides a lower estimate on the Dirichlet and Neumann traces of $u_j$ on $A$. This proposition in fact holds for a wider class of convex polygons, and so in Section \ref{sec:proof} we will prove a more general version of this proposition (see Theorem \ref{thm:Main}). We also have an upper bound on $R_2-R_1$.
\begin{lem} \label{lem:difference}
{Let $V_1, V_2$ satisfy  Assumption \ref{assum:potential-main}, and let $k>0$.} 
There exists a constant $C>0$, depending only on the constant in \eqref{eqn:constants}, such that
\begin{align*}
\norm{(R_2-R_1)f}_{\mathcal{H}^1_k(A)} \leq C\norm{V_2-V_1}_{L^{\infty}(S)}\norm{kf}_{L^2(A)}.
\end{align*}
Moreover, $R_2-R_1:L^2(A) \to \mathcal{H}^1_k(A)$ is a compact operator.
\end{lem}

We next  prove Theorems \ref{thm:injective} and \ref{thm:Main-W} using Proposition \ref{prop:Main} and Lemma \ref{lem:difference}.  
\begin{proof1}{Theorem \ref{thm:injective}}
We first show that {$I-R_1R_2:L^2(A)\to L^2(A)$} is injective. Suppose that $f\in L^2(A)$ with $(I-R_1R_2)f=0$. Let $u_2$ be the  solution of \eqref{eqn:u-i} with {$V=V_2$, $G=0$, and with $g=f$ on $A$, $g=0$ on $S\backslash A$.} Also, let $w_1$ be the solution of \eqref{eqn:u-i} with {$V=V_1$, $G=0$, $g=R_2f$ on $A$, $g=0$ on $S\backslash A$}. We then define the function $v$ on the rectangle {$\Omega = [0,2]\times[0,1]$}
by
$$
v(x,y) = \begin{cases}
w_1(x,y) &\text{ for } {(x,y) \in [0,1]\times [0,1]},\\
-u_2(2-x,y) &\text{ for }{ (x,y) \in [1,2]\times [0,1]}.
\end{cases}
$$
Then, recalling the relation between $V_1$ and $V_2$ and the function $V$ {stated in Assumption \ref{assum:potential-main}}, $v(x,y)$ satisfies 
$$
\begin{cases}
 (\Delta +k^2V)v=0 & \text{ in } \Omega \backslash A,\\
 \pa_{\nu}v+ikv = 0 &\text{ on } \pa \Omega,
\end{cases}
$$
where we continue to write  $A=\{1\}\times [0,1]$.

 By the uniqueness of solutions to the impedance problem (see Proposition 8.1.3 in \cite{melenk1995generalized}), to conclude that $f=0$ it is sufficient to show that $v$ and $\pa_{x}v$ are continuous across the line {$A$}.

By the definitions of $u_2$, $w_1$, and $R_2$, we have {that on $A$}
\begin{align} \label{eqn:injective1}
    \pa_{x}u_2 - iku_2 = R_2f =  \pa_{x}w_1 + ikw_1.
\end{align}
Moreover,{ since we are assuming that $(I-R_1R_2)f = 0$ on $A$}, we also have {the equality on $A$ of}
\begin{align} \label{eqn:injective2}
     \pa_{x}u_2 + iku_2 = f ={R_1R_2 f}=\pa_{x}w_1 - ikw_1 .
\end{align}
Combining equations \eqref{eqn:injective1} and \eqref{eqn:injective2} implies that $\pa_xu_2 =\pa_x w_1$ and $u_2 = -w_1$ on $A$. This ensures that $v$ and $\pa_xv$ are continuous across the line {$A$}. Hence, $f=0$ and so  $I-R_1R_2$ is injective.

To prove that $I-R_1R_2:L^2(A) \to \mathcal{H}^1_k(A)$ is bijective we use the decomposition given in \eqref{eqn:decomposition}. By Proposition \ref{prop:Main} the operator $T := (I-R_1)(I+R_2): L^2(A) \to \mathcal{H}^1_k(A)$ is invertible {and $T^{-1}$ is bounded}. Therefore, using \eqref{eqn:decomposition} we have
\begin{align}\label{eqn:Tinverse}
    T^{-1}(I-R_1R_2) = I - T^{-1}(R_2 - R_1).
\end{align}
{Note that $R_2 - R_1:L^2(A) \to \mathcal{H}^1_k(A)$ is compact by Lemma \ref{lem:difference}. Hence, since $T^{-1}:\mathcal{H}^1_k(A)\to L^2(A) $ is bounded, $T^{-1}(R_2 - R_1):L^2(A) \to L^2(A)$ is compact}. {Using \eqref{eqn:Tinverse}} this implies that $T^{-1}(I-R_1R_2):L^2(A) \to L^2(A)$ is a Fredholm operator of index $0$, and also has trivial kernel. Therefore, the range of $T^{-1}(I-R_1R_2)$ is  $L^2(A)$, and so $I-R_1R_2: L^2(A) \to \mathcal{H}^1_k(A)$ is bijective.

{We also note that combining Lemma \ref{lem:R-boundedness}, Lemma \ref{lem:difference}, and \eqref{eqn:decomposition},  the operator $I-R_1R_2: L^2(A) \to \mathcal{H}^1_k(A)$ is bounded}.   The fact that its inverse $W$ is bounded therefore follows from the bounded inverse theorem.  
\end{proof1}


\begin{proof1}{Theorem \ref{thm:Main-W}}
From \eqref{eqn:decomposition}, with $f = Wg$ we have
\begin{align*}
    \norm{(I-R_1)(I+R_2)Wg}_{\mathcal{H}^1_k(A)} \leq \norm{g}_{\mathcal{H}^1_k(A)} + \norm{(R_2-R_1)Wg}_{\mathcal{H}^1_k(A)}.
\end{align*}
Applying Proposition \ref{prop:Main} and Lemma \ref{lem:difference} thus gives
\begin{align*}
   c^*c_{\delta}^*k(1+k)^{-3(1+\delta)} \norm{Wg}_{L^2(A)} \leq \norm{g}_{\mathcal{H}^1_k(A)} + C\norm{V_2-V_1}_{L^{\infty}(S)}\norm{kWg}_{L^2(A)}.
\end{align*}
{Here, $c^*$, $c^*_\delta$, and $C$, are positive constants that depend only on  $\delta$ and the constant in \eqref{eqn:constants}.}
{In particular, for $\norm{V_2-V_1}_{L^{\infty}(S)}\leq \eps_\delta(1+k)^{-3(1+\delta)}$,
\begin{align*}
   \norm{Wg}_{L^2(A)} \leq \Big( c^*c_{\delta}^*k(1+k)^{-3(1+\delta)} - C \eps_{\delta}k(1+k)^{-3(1+\delta)}\Big)^{-1}\norm{g}_{\mathcal{H}^1_k(A)}.
\end{align*}
The estimate in the theorem follows from choosing $\eps_\delta>0$  such that $C \eps_\delta \leq \tfrac{1}{2}c^*c^*_\delta$.
}
\end{proof1}

It remains to prove Proposition \ref{prop:Main} and Lemma \ref{lem:difference}. In Section \ref{sec:proof} we prove {Theorem \ref{thm:Main} which is} a more general version of Proposition \ref{prop:Main} for a class of convex polygons.  We end this section by proving Lemma \ref{lem:difference} which follows in a straightforward manner from Proposition \ref{prop:elliptic1}.
\begin{proof1}{Lemma \ref{lem:difference}}
Let $u_1$, $u_2$ be the solutions of \eqref{eqn:u-i} {with potentials $V_1, V_2$ respectively, and with $g=f$ and $G=0$.} Then, setting $v := u_2 - u_1$ we have 
\begin{equation}\label{eqn:R difference}
(R_2 - R_1)f =\big( \pa_{\nu}v - ik v\big)\big|_{A} = -2ikv\big|_{A}.
\end{equation}
Since, 
$$
\begin{cases}
    (\Delta + k^2V_2)v  = k^2(V_1 - V_2)u_1  &\text{ in } S, \\
    \pa_{\nu}v+ikv  = 0  &\text{ on } \pa S.
\end{cases}
$$
Proposition \ref{prop:elliptic1} {with $V=V_2$, $G= k^2(V_1 - V_2)u_1,$ and $g=0$ yields}
\begin{align*}
    \norm{v}_{\mathcal{H}^1_k(S)}  + \norm{kv}_{L^2(\pa S)} + \norm{\pa_\nu v}_{L^2(\pa S)} + \norm{\pa_\tau v}_{L^2(\pa S)} \leq C\norm{k^2(V_1-V_2)u_1}_{L^2(S)}.
\end{align*}
Applying Proposition \ref{prop:elliptic1} again also ensures that $\norm{k u_1}_{L^2(S)} \leq C\norm{f}_{L^2(A)}$, and so the claimed estimate on the norm of $(R_2-R_1)f$ follows {from  \eqref{eqn:R difference}}.

In addition, these estimates ensure that $\Delta v$ is in $L^2(S)$, with $\pa_{\nu}v{\big|_{\pa S}}\in H^1(\pa S)$. Since $S$ is a convex polygon, by Theorem 1.6.1.5 in \cite{grisvard2011elliptic}, there exists a function $w$ in $H^2(S)$ with normal derivative $\pa_{\nu}v$ on $\pa S$. The function $v-w$ satisfies Neumann boundary conditions on $\pa S$, with $\Delta(v-w)\in L^2(S)$, and as $S$ is convex, applying Theorem 4.3.1.4 in \cite{grisvard2011elliptic} to $v - w$ ensures  that $v\in H^2(S)$. Another application of Theorem 1.6.1.5 then gives $v{\big|_A}\in H^{3/2}(A)$.  Therefore, for each $k>0$ fixed, $R_2-R_1$ is bounded from $L^2(A)$ to $H^{3/2}(A)$, and so for each fixed $k$, $R_2-R_1$ is a compact operator from $L^2(A)$ to $\mathcal{H}^1_k(A)$.
\end{proof1}

\section{Bounds on $I \pm R$ in the general setting} \label{sec:proof}

In this section we complete the proofs of the main theorems by proving Proposition \ref{prop:Main}. Since this proposition in fact holds in more generality, we work in the following setting. 

Let $\Omega \subset \R^2$ be a convex polygon and $A \subset \pa \Omega$ a side of the polygon. Given $f\in L^2(A)$, $k>0$, and $V \in C^1(\Omega)$, suppose that $u$ solves the elliptic problem
\begin{equation}\label{eqn:u}
\begin{cases}
(\Delta + k^2V)u  = 0& \text{ in } \Omega, \\ 
\pa_{\nu} u + iku  = 0& \text{ on } \pa \Omega\backslash A, \\
\pa_{\nu} u + iku  = f &\text{ on } A.
\end{cases}
\end{equation}
Here $\nu$ is the outward pointing unit normal to $\pa \Omega$. Analogously to Definition \ref{defn:R-Omega}, 
{for $k>0$ we define the operator $R$ on $L^2(A)$  by} 
\begin{align}\label{eqn:R def}
R f = \big( \pa_{\nu}u - iku \big)\big|_{A}.
\end{align}


We will again use the modified spaces, $\norm{\cdot}_{\mathcal{H}^1_k(\Omega)}$ and $\norm{\cdot}_{\mathcal{H}^1_k(A)}$ from \eqref{defn:H1-A} (with the square $S$ replaced by $\Omega$). 
We now state the main result of this section.


\begin{thm} \label{thm:Main}
Let $(a_0,b_0)$ be a vertex on the side $A$ and let  $V {\in C^1(\Omega)}$ be a {non-negative}, real valued potential, for which  there exists $c>0$ such that 
\begin{align}\label{eqn:non-trapping}
 2V(x,y) + (x-a_0,y-b_0)\!\cdot\!\nabla V(x,y) \geq c, \qquad (x,y)\in\Omega,
\end{align}
i.e. in the language of Assumption \ref{assum:potential-main}, $V$ is non-trapping with respect to $(a_0,b_0)$.
Then, given $\delta>0$, there exist constants $c^*>0$, $c^*_{\delta}>0$, such that for all $k>0$ and $R$ as defined in \eqref{eqn:R def}
\begin{align*}
\norm{(I-R)f}_{\mathcal{H}^1_k(A)} & \geq c^*\norm{kf}_{L^2(A)}, \\
\norm{(I+R)f}_{L^2(A)} & \geq c_{\delta}^*(1+k)^{-3(1+\delta)}\norm{f}_{L^2(A)}, 
\end{align*}
for all $f\in L^2(A)$. The constants $c^*$ and $c^*_{\delta}$ depend additionally only on the diameter and inner radius of the polygon $\Omega$ and the interior angles at the vertices on $A$.

In addition, the images of $(I-R)$ and $(I+R)$ on $L^2(A)$ are $\mathcal{H}^1_k(A)$ and $L^2(A)$ respectively.
\end{thm}
In the special case that $\Omega$ is the unit square and $A$ is the side $\{1\}\times[0,1]$, this theorem implies Proposition \ref{prop:Main}. For the rest of this section we prove the theorem. The main part of the proof is to establish the lower bound estimates, and then we end the proof by determining the ranges of $I\pm R$.

\begin{proof1}{Theorem \ref{thm:Main}}
 Let $f\in L^2(A)$. Without loss of generality we assume that  $\norm{f}_{L^2(A)}=1$. After a rotation, reflection, and dilation, we also assume that $A = \{1\}\times[0,1]$, $\Omega$ is contained within the set $\{(x,y)\in\mathbb{R}^2:x \leq 1\}$ and the vertex $(a_0,b_0) = (1,0)$.
First, we claim that estimates in Theorem \ref{thm:Main} are a consequence of the following two propositions.
\begin{prop} \label{prop:weak}
Let $u$ solve \eqref{eqn:u}. There exists  $C>0$, independent of $k$, such that
\begin{align*}
\norm{\pa_\nu u}_{L^2(A)} \leq  C\norm{u}_{\mathcal{H}^1_k(A)}^{1/4}.
\end{align*}
\end{prop}
\begin{prop} \label{prop:weak2}
Let $u$ solve \eqref{eqn:u}, and let $\delta>0$ be given. There exists  $C_\delta>0$, independent of $k$, such that
\begin{align*}
 \norm{ku}_{L^2(A)} \leq C_{\delta}(1+k)^{3(1+{\delta})/2}\norm{\pa_\nu u}_{L^2(A)}^{1/2}.
\end{align*}
\end{prop}
To see that the estimates in Theorem \ref{thm:Main} follow from these two propositions, note that using the definition of $R$, we have
\begin{align*}
\tfrac{1}{2k}\norm{(I-R)f}_{\mathcal{H}^1_k(A)} = \norm{u}_{\mathcal{H}^1_k(A)}, \qquad \tfrac{1}{2}\norm{(I+R)f}_{L^2(A)}= \norm{\pa_\nu u}_{L^2(A)}.
\end{align*}
Therefore, to prove the estimates in Theorem \ref{thm:Main}  we  need to show that there exist $c^*>0$ and $c_\delta>0$, independent of $k$, such that
\begin{equation}
    \label{claimest}
\norm{u}_{\mathcal{H}^1_k(A)}\geq \tfrac{1}{2}c^* \qquad\text{and}\qquad  \norm{\pa_\nu u}_{L^2(A)} \geq \tfrac{1}{2}c_{\delta}^*(1+k)^{-3(1+\delta)}.
\end{equation}
Now, since $\norm{f}_{L^2(A)} = 1$, we have $$1 \leq \norm{ku}_{L^2(A)} + \norm{\pa_\nu u}_{L^2(A)}\leq \norm{u}_{\mathcal{H}^1_k(A)} + \norm{\pa_\nu u}_{L^2(A)}.$$ Therefore, the estimates in the Propositions above yield
\begin{align*}
1 \leq \norm{u}_{\mathcal{H}^1_k(A)} + C\norm{u}_{\mathcal{H}^1_k(A)}^{1/4},\qquad \qquad 1 \leq C_{\delta}(1+k)^{3(1+{\delta})/2}\norm{\pa_\nu u}_{L^2(A)}^{1/2} + \norm{\pa_\nu u}_{L^2(A)}.
\end{align*}
Rearranging these inequalities  proves the claimed estimates in \eqref{claimest}. We prove Propositions \ref{prop:weak} and \ref{prop:weak2} in Sections \ref{sec:weak} and \ref{sec:weak2} respectively. 

With Propositions \ref{prop:weak} and \ref{prop:weak2} in place, it only remains to find the ranges of $I\pm R$. We first consider $I+R$: By Lemma \ref{lem:elliptic} below, this operator maps $L^2(A)$ into itself, and so we need to show that the range of $I+R$ contains $L^2(A)$. Given $g\in L^2(A)$, we let $v$ be the unique $H^1(\Omega)$ weak solution to the elliptic problem
$$
\begin{cases}
(\Delta + k^2V)v  = 0  &\text{ in } \Omega, \\
\pa_{\nu}v+ikv  = 0  &\text{ on } \paOA, \\
\pa_{\nu}v  = \tfrac{1}{2}g & \text{ on } A.
\end{cases}
$$
By the trace theorem, in particular $kv\in L^2(A)$. Setting $f$ to be $\tfrac{1}{2}g + ik v|_{A} \in L^2(A)$, we therefore have $(I+R)f = 2\pa_{\nu}v|_{A} = g$. This means that the range of $I+R$ on $L^2(A)$ is given by $L^2(A)$.
\\
\\
We now turn to $I-R$.  We first record the analogous elliptic estimates to Proposition \ref{prop:elliptic1} that are satisfied by the solution $u$ in \eqref{eqn:u}, with an identical proof to that of Proposition \ref{prop:elliptic1}. From now on we write $\ud s$ to denote the measure on $\partial \Omega$ induced by the Euclidean measure $\dv=dxdy$ in $\R^2$.

\begin{lem}\label{lem:elliptic}
The elliptic equation with boundary conditions in \eqref{eqn:u} has a unique solution $u\in H^1(\Omega)$, and there exists  $C>0$, independent of $k$, such that
\begin{align*}
\norm{u}_{\mathcal{H}^1_k(\Omega)}  + \norm{ku}_{L^2(\pa \Omega)} + \norm{\pa_\nu u}_{L^2(\pa \Omega)} + \norm{\pa_\tau u}_{L^2(\pa \Omega)} \leq C\norm{f}_{L^2(A)}.
\end{align*}
Here $\tau$ is the unit tangent vector to $\pa \Omega$. 
\end{lem}

By Lemma \ref{lem:elliptic}, the operator $I-R:L^2(A)\to\mathcal{H}^1_k(A)$ is bounded, and we will show that its image contains $\mathcal{H}^1_k(A)$. Given $g\in \mathcal{H}_k^1(A)$, let $w$ be the unique $H^1(\Omega)$ weak solution to the elliptic problem
$$
\begin{cases}
(\Delta + k^2V)w  = 0 &\text{ in } \Omega, \\
\pa_{\nu}w+ikw  = 0 &\text{ on } \paOA, \\
ikw  = \tfrac{1}{2}g &\text{ on } A.
\end{cases}
$$
Again, by the trace theorem,  $w\in H^{1/2}(\pa\Omega)$. The normal derivative of $w$ is in $L^2(\paOA)$, while $w$ is itself in $H^1(A)$. Since $\Omega$ is a convex polygon, the interior angles where $A$ meets its adjacent sides is strictly less than $\pi$. Therefore, by the estimates on solutions to elliptic equations with mixed Dirichlet-Neumann boundary conditions given in \cite{brown-mixed}, $\nabla w$ is in $L^2(\pa\Omega)$. In particular, $\pa_{\nu}w|_{A}\in L^2(A)$. Setting $f$ to be $\pa_{\nu}w|_{A} + \tfrac{1}{2}g \in L^2(A)$, we therefore have $(I-R)f = 2ikw = g$. This means that the range of $I-R$ on $L^2(A)$ is given by $\mathcal{H}_k^1(A)$, and this completes the proof of Theorem \ref{thm:Main}.
\end{proof1}
\subsection{Proof of Proposition \ref{prop:weak}}\label{sec:weak}

Using Lemma \ref{lem:elliptic} above, it is straightforward to show that $\norm{u}_{\mathcal{H}^1_k(A)}$ and $\norm{\pa_\nu u}_{L^2(A)}$  provide control on the boundary data of $u$ on $\pa \Omega \backslash A$ in the following sense.
\begin{lem} \label{lem:control}
There exists a constant $C>0$, independent of $k$, such that {for $u$ solving \eqref{eqn:u},}
\begin{align*}
\norm{ku}_{L^2(\paOA)} + \norm{\pa_{\nu}u}_{L^2(\paOA)} \leq C\min\{\norm{u}_{\mathcal{H}^1_k(A)}^{1/2},\norm{\pa_\nu u}_{L^2(A)}^{1/2}\}. 
\end{align*}
\end{lem}
\begin{proof1}{Lemma \ref{lem:control}}
We use the weak formulation of \eqref{eqn:u}, which states that
\begin{align} \label{eqn:weak}
0 = -\int_{\Omega}\nabla u \cdot\nabla \bar{v} \,\dv+ k^2\int_{\Omega}Vu\bar{v} \,\dv + \int_{\pa \Omega} \pa_{\nu}u\bar{v} \,\ud s
\end{align}
for all $v\in H^{1}(\Omega)$. Setting $v = u$ and using the boundary conditions yields
\begin{align*}
0 = -\int_{\Omega}|\nabla u|^2 \,\dv + k^2\int_{\Omega}V|u|^2\,\dv - ik \int_{\paOA}|u|^2 \ud s + \int_{A}\pa_\nu u\bar{u} \ud s.
\end{align*}
{Taking the imaginary part of the equation above, and using that $iku = -\pa_{\nu}u$ on $\paOA$,} we have
\begin{align*}
  \norm{\pa_{\nu}u}_{L^2(\paOA)}^2=\norm{ku}_{L^2(\paOA)}^2
= k\text{Im}\int_{A}\pa_\nu u \bar{u} \ud s\leq \norm{\pa_\nu\, u}_{L^2(A)} \norm{ku}_{L^2(A)}.
\end{align*}
The result follows from Lemma \ref{lem:elliptic}. 
\end{proof1}

We will see below that Lemma \ref{lem:control} provides sufficient control on $\norm{ku}_{L^2(A)}$ in terms of $\norm{\pa_\nu u}_{L^2(A)}$ in order to prove Proposition \ref{prop:weak2}.  First, we use Lemma \ref{lem:control} together with the elliptic estimates in Lemma \ref{lem:elliptic} to immediately bound $\norm{\pa_\nu u}_{L^2(A)}$ in terms of $\norm{u}_{\mathcal{H}^1_k(A)}$ and obtain Proposition \ref{prop:weak}. To do this we use the weak formulation again with a different choice of test function $v$.

\begin{proof1}{Proposition \ref{prop:weak}}
We prove this proposition by choosing appropriate test functions  in \eqref{eqn:weak}, following the same ideas as for the original regularity estimates of the impedance problem in Proposition~8.1.4 in \cite{melenk1995generalized}. 
{Let ${\bf z}:=(x-a, y-b)$ with $(a,b)$ to be specified later in the proof, and let $v={\bf z}\cdot \nabla u$.} 
We shall use that 
\begin{align*}
\text{Re}(u\bar{v}) & = \tfrac{1}{2}(x-a)\pa_x\left(|u|^2\right)+ \tfrac{1}{2}(y-b)\pa_y\left(|u|^2\right), \\
\text{Re}(\nabla u \cdot \nabla \bar{v}) & = |\nabla u|^2 + \tfrac{1}{2}(x-a)\pa_x\left(|\nabla u|^2\right) + \tfrac{1}{2}(y-b)\pa_y\left(|\nabla u|^2\right).
\end{align*}
Integrating by parts in \eqref{eqn:weak}, using the above equations and \eqref{eqn:u}, gives 
\begin{align*}
0 & =  - k^2\int_\Omega V|u|^2 \dv - \tfrac{1}{2}k^2\int_\Omega |u|^2 ({\bf z}\cdot \nabla V)\dv  - \tfrac{1}{2}\int_{\pa \Omega}({\bf z}\cdot \nu)|\nabla u|^2\ud s \\
& +\tfrac{1}{2}k^2\int_{\pa \Omega} ({\bf z}\cdot \nu) V|u|^2 \ud s+{\text{Re}\Big(}- ik\int_{\paOA}u ({\bf z}\cdot\nabla \bar{u})\ud s + \int_{A}\pa_x u ({\bf z}\cdot \nabla \bar{u})\ud s\Big).
\end{align*}
Therefore, there exists $C>0$, depending only on the diameter of $\Omega$, such that
\begin{align*}
 &\Big|- k^2\int_\Omega V|u|^2\dv - \tfrac{1}{2}k^2\int_\Omega |u|^2({\bf z}\cdot \nabla V) \dv- \tfrac{1}{2}\int_{\pa \Omega}({\bf z}\cdot \nu)|\nabla u|^2 \ud s +  \int_{A}\pa_x u(x-a)\pa_{x}\bar{u} \ud s\Big|\\
 &\qquad \leq C \Big(\|ku\|^2_{L^2(\pa \Omega)}+ \|ku\|_{L^2(\paOA)}\|\nabla u\|_{L^2(\paOA)} +\|\partial_\tau u\|_{L^2(A)}\|\partial_\nu u\|_{L^2(A)}\Big).
\end{align*}
Using Lemma \ref{lem:control}, {and that $\norm{u}_{\mathcal{H}^1_k(A)}$ is bounded,} we can rearrange this as
\begin{align*}
\left|- k^2\int_\Omega \left(V +\tfrac{1}{2}({\bf z}\cdot \nabla V)\right)|u|^2\dv- \tfrac{1}{2}\int_{\pa \Omega}({\bf z}\cdot \nu)|\nabla u|^2 \ud s +  \int_{A}(x-a)|\pa_{x}{u}|^2 \ud s\right| \leq  C\norm{u}_{\mathcal{H}^1_k(A)}^{1/2}.
\end{align*}
Expanding the second integral on the left hand side into the parts on $\paOA$ and $A$, {using that ${\bf z}\cdot \nu=(x-a)$ on $A$}, and incorporating more terms on the right hand side, we can rewrite the above as
\begin{align} \label{eqn:weak1}
\left|- k^2\int_\Omega \left(V +\tfrac{1}{2}({\bf z}\cdot \nabla V)\right)|u|^2\dv -\tfrac{1}{2}\int_{\paOA}({\bf z} \cdot \nu)|\nabla u|^2 \ud s  +\tfrac{1}{2}\int_{A}({\bf z} \cdot\nu)|\pa_{x}{u}|^2 \ud s\right| \leq  C\norm{u}_{\mathcal{H}^1_k(A)}^{1/2}.
\end{align}
To conclude the proof, we need to choose $(a,b)$ such that each term on the right hand side of \eqref{eqn:weak1} is strictly negative. To do this we first note that for each $\eps>0$, there exists $c_0 = c_0(\eps)>0$ such that
\begin{align*}
(x-1,y-\eps)\cdot \nu \geq c_0>0 \text{ on } \paOA, \qquad (x-1,y-\eps)\cdot \nu = 0 \text{ on } A.
\end{align*}
Therefore, we set ${\bf z} = (x-1-\tfrac{c_0}{2},y-\eps)$, and this has the property that ${\bf z}\cdot\nu \geq \tfrac{c_0}{2}$ on $\paOA$, while ${\bf z}\cdot \nu = -\tfrac{c_0}{2}<0$ on $A$. Finally, for $\eps>0$ sufficiently small, and decreasing $\eps$ and $c_0>0$ if necessary, by the assumption on the potential $V$ in \eqref{eqn:non-trapping}, we have the lower bound $2V +{\bf z}\cdot \nabla V\geq
\tfrac{1}{2}c$ on $\Omega$. This is because the non-trapping assumption on $V$ in \eqref{eqn:non-trapping} also holds for all points $(a,b)$ sufficiently close to $(a_0,b_0) = (1,0)$, with $c$ replaced by $\tfrac{1}{2}c$.  Therefore, each term on the left hand side of \eqref{eqn:weak1} has the same sign, and is each individually bounded by $C\norm{u}_{\mathcal{H}^1_k(A)}^{1/2}$, concluding the proof. 
\end{proof1}

\subsection{Proof of Proposition \ref{prop:weak2}}\label{sec:weak2}

{We start the proof by noting that we may assume, without loss of generality, that $\norm{\pa_\nu u}_{L^2(A)}\leq \norm{\pa_\nu u}_{L^2(A)}^{1/2}$. Indeed, if this were not the case, then  $1\leq \norm{\pa_\nu u}_{L^2(A)}$ and we would be done. It then follows that by Lemma \ref{lem:control}
\begin{align} \label{eqn:g1}
\norm{\partial_\nu u}_{L^2(\pa \Omega)} + \norm{ku}_{L^2(\paOA)}  & \leq C\norm{\pa_\nu u}_{L^2(A)}^{1/2}.
\end{align}
Therefore, we wish to show that the estimate in \eqref{eqn:g1} implies that
\begin{align} \label{eqn:Main1a}
\norm{ku}_{L^2(A)} \leq C_{\delta}(1+k)^{3(1+\delta)/2}\norm{\pa_\nu u}_{L^2(A)}^{1/2}.
\end{align}
}
Let $\{\la_m\}_{m\geq0}$ be the eigenvalues of the problem
$$
\begin{cases}
(\Delta + \la_m V)w_m  = 0 & \text{ in } \Omega \\
\pa_{\nu}w_m = 0 & \text{ on } \pa\Omega,
\end{cases}
$$
with corresponding orthonormal eigenfunctions $w_m$ (with respect to the $(V\cdot,\cdot)$-inner product on $\Omega$). In particular $\la_0=0$ and $w_0$ is constant in $\Omega$.

For $k^2{\notin \{\la_m\}_{m\geq0}}$, let $G(\x;\x')$ be the Neumann Green's function for $\Delta + k^2V$. That is,
\begin{align}\label{defn:NGreen}
G(\x;\x') = \sum_{m=0}^{\infty}\frac{1}{k^2-\la_m} w_m(\x)w_m(\x').
\end{align}
{Here, we adopt the notation $\x = (x,y)$, $\x' = (x',y')$ for points in $\R^2$.}

To prove \eqref{eqn:Main1a} we will split into two cases: $k$ large and small. 
Indeed, let  $c^*>0$ be such that $2c^*$ is a lower bound for the spectral gap. That is, 
$$\lambda_1 - \lambda_0 = \lambda_1 \geq 2c^*.$$ 
In Case 1 below we deal with values of $k$ such that $k^2 \leq c^*$, and in Case 2 we deal with $k^2>c^*$.
\\

\noindent \underline{\textbf{Case 1:} ($k^2$ small)}\\
\\
{For this case we assume that $k^2 \leq c^*$. } Since $\lambda_0 = 0$ is a Neumann eigenvalue, with corresponding constant eigenfunction $w_0$, we first subtract a constant $u_0$ from $u$, so that $u$ is orthogonal to $u_0$ in the $(V\cdot,\cdot)$ inner product. Then, setting $\tilde{u} = u-u_0$, it satisfies
\begin{equation} \label{eqn:case1aa}
\begin{cases}
(\Delta + k^2V)\tilde{u} = -k^2Vu_0 &\text{ in } \Omega,\\
\pa_{\nu}\tilde{u} = \partial_{\nu}u &\text{ on } \pa\Omega.
\end{cases}
\end{equation}

Let $G_0(\x;\x')$ be the part of the Green's function $G$ from \eqref{defn:NGreen} which is orthogonal to the constant eigenfunction $w_0$ in the $(V\cdot,\cdot)$ inner product. That is, 
\begin{align*}
G_0(\x;\x') = \sum_{m=1}^{\infty}\frac{1}{k^2-\la_m} w_m(\x)w_m(\x').
\end{align*}
Then, since 
$
-k^2u_0 \int_{\Omega} {G}_0(\x;\x') V(\x') \ud \x'= 0,
$
we have 
\begin{align} \label{eqn:case1a}
\tilde{u}(\x) =  -\int_{\pa\Omega} {G}_0(\x;\x') \partial_{\nu}u(\x') \ud s(\x').
\end{align}
To bound the right hand side of \eqref{eqn:case1a}, we first consider 
\begin{align*}
F_0(\x') = -\int_{\Omega} G_0(\x;\x') g(\x)\ud \x,
\end{align*}
where $g$ is $L^2(\Omega)$-normalized. Using $\la_m-k^2 \geq c^*$ for $m\geq1$, we have $\norm{F_0}_{L^2(\Omega)} \leq \left(c^*\right)^{-1}$, and $F_0$ satisfies the equation
\begin{align*}
(\Delta + k^2V)F_0 = -\Pi_0g \text{ in } \Omega, \qquad \pa_{\nu}F_0 = 0 \text{ on } \pa\Omega.
\end{align*}
Here $\Pi_0$ is the projection operator to the orthogonal complement of the constant eigenfunction. This means that $\norm{\nabla F_0}_{L^2(\Omega)} \leq C\left(c^*\right)^{-1/2}$, and by a Sobolev trace estimate, the same holds for $\norm{F_0}_{L^2(\pa\Omega)}$.  Returning to the expression in \eqref{eqn:case1a}, we therefore have that 
\begin{align*}
\int_{\Omega}\tilde{u}(\x)g(\x) \ud \x = \int_{\pa\Omega}F_0(\x')\partial_{\nu}u(\x') \ud s(\x')
\end{align*}
can be bounded in absolute value by 
\begin{align*}
\norm{F_0}_{L^2(\pa\Omega)}\norm{\partial_{\nu}u}_{L^2(\pa\Omega)} \leq  C\left(c^*\right)^{-1/2}\norm{\pa_\nu u}_{L^2(A)}^{1/2}.
\end{align*}
Here we have also used the estimate on $\partial_{\nu}u$ from \eqref{eqn:g1}. Therefore, by duality,
\begin{align*}
\norm{\tilde{u}}_{L^2(\Omega)} \leq C\left(c^*\right)^{-1/2}\norm{\pa_\nu u}_{L^2(A)}^{1/2}.
\end{align*}
Combining this with the equation in \eqref{eqn:case1aa} thus gives
\begin{align*}
\norm{\tilde{u}}_{L^2(\pa\Omega)}^2 \leq C\norm{\tilde{u}}^2_{H^1(\Omega)}  \leq C\left(c^*\right)^{-1}\norm{\pa_\nu u}_{L^2(A)} + C\left|\int_{\pa\Omega} \tilde{u}\partial_{\nu}u \ud s\right|.
\end{align*}
Therefore, using \eqref{eqn:g1} again we have
\begin{align*} 
|u_0| - \norm{{u}}_{L^2(\paOA)} \leq \norm{\tilde{u}}_{L^2(\pa\Omega)} \leq C\left(c^*\right)^{-1/2}  \norm{\pa_\nu u}_{L^2(A)}^{1/2}.
\end{align*}
The estimate $\norm{{ku}}_{L^2(\paOA)}\leq C\norm{\pa_\nu u}_{L^2(A)}^{1/2}$ from \eqref{eqn:g1}, thus implies  the bound
\begin{align*}
|u_0| \leq Ck^{-1}\left(c^*\right)^{-1/2} \norm{\pa_\nu u}_{L^2(A)}^{1/2}.
\end{align*}
Therefore,
\begin{align} \label{eqn:case1b}
\norm{ku}_{L^2(A)} \leq \norm{k\tilde{u}}_{L^2(A)} + Ck|u_0| \leq C\left( c^*\right)^{-1/2} \norm{\pa_\nu u}_{L^2(A)}^{1/2},
\end{align}
and this implies the claim in \eqref{eqn:Main1a}.\ \\

\noindent\underline{\textbf{Case 2:} ($k^2$ bounded away from zero)}
\\ 
\\
{For this case we assume that $k^2 > c^*$.} To deal with this, we need to consider  Neumann eigenmodes with frequencies centered at $k^2$. In particular, we want to rule out a Neumann eigenfunction having small Dirichlet data on $\paOA$ relative to the Dirichlet data on $A$. To do this we prove the following proposition (using analogous techniques to those for triangles in \cite{christianson2017triangle}, \cite{christianson2019triangle}).
\begin{lem} \label{prop:efnN}
{Let $V \in C^1(\Omega)$ be as in Theorem \ref{thm:Main}.}
Let $w$ be $L^2(\Omega)$-normalized and satisfy
\begin{align*}
(\Delta + k^2 V)w = h \emph{ in } \Omega, \quad \pa_{\nu}w = 0 \emph{ on } \pa \Omega,
\end{align*}
for some function $h\in H^1(\Omega)$, and $k^2 >c^*$. Then, there exists a constant $c>0$, independent of $k$,
such that, if $\norm{h}_{\mathcal{H}^1_k(\Omega)} \leq ck^2$, then
\begin{align*}
\int_{\paOA} |w|^2 \ud s\geq c.
\end{align*}
\end{lem}
\begin{rem} \label{rem:efnDN}
In the course of proving Lemma \ref{prop:efnN} we will in fact show that the lower bound holds for the Dirichlet trace on the part of $\pa\Omega$ complement to any two adjacent sides of the convex polygon.
\end{rem}
\begin{proof1}{Lemma \ref{prop:efnN}}
The idea of the proof is to integrate the equation for $w$ against an appropriately chosen test function. For the vector field $X = (x-1)\pa_x + y\pa_y$, we have
\begin{align*}
\left[\Delta,X\right] = 2\Delta, \qquad \left[V,X\right] = -X(V) .
\end{align*}
 Therefore,
\begin{align}  \nonumber
\int_{\Omega}\Delta(Xw)\bar{w} - (Xw)(\Delta \bar{w}) \dv & = \int_{\Omega}X(\Delta w)\bar{w}  - (Xw)\Delta \bar{w} +2(\Delta w)\bar{w} \dv\\ \label{eqn:efnN1}
&= -k^2\int_{\Omega}\left(2V + X(V)\right)|w|^2\dv + \int_{\Omega} X(h)\bar{w} - (Xw)\bar{h} + 2h\bar{w}\dv.
\end{align}
We first obtain an upper bound on the right hand side of \eqref{eqn:efnN1}. Using the assumption \eqref{eqn:non-trapping} on $V$, together with the $L^2(\Omega)$-boundedness of $w$, the first integral on the right hand side is at most $-c_1k^2$. For the second integral, we first note that $\norm{w}_{\mathcal{H}^1_k(\Omega)}$ is bounded by $Ck$. Therefore, by choosing the constant $c$ in the statement of the lemma to be sufficiently small, we can bound the second integral from above by $\tfrac{1}{2}c_1k^2$.
This choice of $c$ thus ensures that 
\begin{equation}\label{eqn:before green}
    \int_{\Omega}\Delta(Xw)\bar{w} - (Xw)(\Delta \bar{w}) \dv \leq - \tfrac{1}{2}c_1k^2.
\end{equation}
We now turn to the left hand side. By Green's identity, $  \int_{\Omega}\Delta(Xw)\bar{w} - (Xw)(\Delta \bar{w}) \dv=\int_{\pa \Omega}\pa_{\nu}(Xw)\bar{w} \ud s- \int_{\pa \Omega} (Xw)\pa_{\nu} \bar{w} \ud s$. Therefore, since the second integral on the right hand side vanishes ($w$ satisfies Neumann boundary conditions), \eqref{eqn:before green} can be written as
\begin{align} \label{eqn:efnN2}
\int_{\pa \Omega}\pa_{\nu}(Xw)\bar{w} \ud s\leq - \tfrac{1}{2}c_1k^2.
\end{align}
Next, we break up $\pa \Omega$ into the pieces $A_1$, $A_2, \ldots A_m$, with $A = A_1$ and $A_j$ the remaining sides of the polygon going counter-clockwise. 
Since $\pa_{\nu} =\pa_{x}$,  $\pa_xw = 0$, and $x-1=0$ on $A$ the contribution to the integral in \eqref{eqn:efnN2} on $A$ vanishes. 
\\
\\
 The contribution to \eqref{eqn:efnN2} from the side $A_j$ is  equal to
\begin{align*}
\int_{A_j}\pa_{\nu}(Xw)\bar{w} \ud s &  = \int_{A_j}\left(\pa_{\nu}w\right)\bar{w} \ud s + \int_{A_j}\left(X\pa_{\nu}w\right)\bar{w} \ud s \\
& = \int_{A_j}  \left(x-1,y\right)\cdot \left(\nabla \pa_{\nu}w\right)\bar{w} \ud s.
\end{align*}
Since $\pa_{\tau}\pa_{\nu}w = 0$ on $A_j$, we can rewrite this as
\begin{align*}
 \int_{A_j}  \left(x-1,y\right)\cdot \nu \left(\pa_{\nu}^2w\right)\bar{w} \ud s.
\end{align*}
As $A_j$ is a side of the convex polygon, and $(1,0)$ is a point on the boundary, the quantity $(x-1,y)\cdot\nu$ is a non-negative constant on $A_j$. Moreover, as $(1,0)$ is the vertex joining the sides $A_1$ and $A_m$, this constant is $0$ for $j = m$, and a strictly positive constant for $2 \leq j \leq m-1$. 
We thus have
\begin{align} \label{eqn:efnN3}
\int_{\pa \Omega}\pa_{\nu}(Xw)\bar{w} \ud s & =  \sum_{j=2}^{m-1}\int_{A_j} (x-1,y)\cdot \nu\left(\pa_{\nu}^2w\right) \bar{w}\ud s.
\end{align} 
Using the equation that $w$ satisfies, we can therefore write \eqref{eqn:efnN2} as
\begin{align*}
  \sum_{j=2}^{m-1}\int_{A_j} (x-1,y)\cdot \nu\left(-\pa_{\tau}^2w - k^2Vw + h\right) \bar{w}\ud s\leq - \tfrac{1}{2}c_1k^2 ,
 \end{align*}
and by integrating by parts along $A_j$ this becomes
\begin{align*}
  \sum_{j=2}^{m-1}\int_{A_j} (x-1,y)\cdot \nu\left(\left|\pa_{\tau}w\right|^2 - k^2V|w|^2+h\bar{w}\right) \ud s
   \leq - \tfrac{1}{2}c_1k^2 .
\end{align*}
In the above, the Neumann boundary conditions and the convexity of $\Omega$ ensure that the gradient of $v_m$ is continuous at each vertex, and hence vanishes there. Therefore, the boundary terms from the integration by parts vanish. 
We have $(x-1,y)\cdot \nu >0$ on $A_j$ for $2 \leq j \leq m-1$, $V$ is bounded below on $\Omega$, and $\norm{h}_{\mathcal{H}_k^1(\Omega)} \leq c k^2$. Therefore, again taking $c>0$ sufficiently small ensures that we must have
\begin{align*}
\sum_{j=2}^{m-1}\int_{A_j} k^2|w|^2 \ud s \geq c_2k^2,
\end{align*}
for some constant $c_2>0$. This completes the proof of the lemma. 
\end{proof1}

We now use Lemma \ref{prop:efnN} to handle Case 2. In this case, we set $\tilde{u} = u - w$, where
\begin{align} \label{eqn:case2}
w = \sum_{|\la_m - k^2|\leq c_0}(Vw_m,u)w_m,
\end{align}
and $c_0 = c_0(k)>0$ will be specified below. Since we are in Case 2, we will in particular choose $c_0$ so that this sum does not contain $m=0$. Then,
\begin{align} \label{eqn:case2a}
(\Delta + k^2V)w =h, \qquad \qquad   {h:=}\sum_{|\la_m - k^2|\leq c_0}(k^2-\lambda_m)(Vw_m,u) Vw_m.
\end{align}
Note that \
\begin{align*}
\norm{h}_{\mathcal{H}^1_k(\Omega)}^2 \leq Ck^2c_0^2\sum_{|\la_m - k^2|\leq c_0}\left|(Vw_m,u)\right|^2= Ck^2c_0^2 \int_{\Omega}V |w|^2 \dv.
\end{align*}
Therefore, we will take $c_0 = c_1k$, with $c_1$ sufficiently small, depending only on $c^*$, so that Lemma \ref{prop:efnN} can be applied. This then implies that
\begin{align} \label{eqn:case2b}
\int_{\paOA}|w|^2 \ud s \geq c^2\int_{\Omega} |w|^2 \dv.
\end{align}
Using \eqref{eqn:case2a}, we see that $\tilde{u} = u-w$ satisfies the equation
\begin{align} \label{eqn:case2c}
(\Delta + k^2V)\tilde{u} = - {h} \;\;\text{ in }\Omega, \qquad \qquad  \pa_{\nu}\tilde{u} = \partial_{\nu}u \;\;\text{ on } \pa\Omega. 
\end{align}
We now let ${G}_k(\x;\x')$ be the part of the Green's function for $\Delta + k^2V$ which is orthogonal (in the $(V\cdot,\cdot)$ inner product) to  {$\{w_m: \;|\la_m - k^2|\leq c_0\}$}. That is, 
\begin{align*}
{G}_k(\x;\x') = \sum_{|\la_m-k^2|>c_0}\frac{1}{k^2-\la_m}w_m(\x)w_m(\x').
\end{align*}
Then, since
\begin{align*}
\int_{\Omega}V(\x') w_m(\x'){G}_k(\x;\x') \ud \x'=0
\end{align*}
for those $w_m$ with $|\la_m-k^2|\leq c_0$, from \eqref{eqn:case2c} we have
\begin{align} \label{eqn:case2d}
\tilde{u}(\x) = -\int_{\pa\Omega} {G}_k(\x;\x')\partial_{\nu}u(\x') \ud s(\x').
\end{align}
Analogously to Case 1, to bound the right hand side of \eqref{eqn:case2d} we first consider
\begin{align*}
F_k(\x') = -\int_{\Omega} G_k(\x;\x') g(\x)\ud \x,
\end{align*}
where $g$ is $L^2(\Omega)$-normalized. Using $|\la_m-k^2| \geq c_0$ for  those $m$ appearing in the sum in the definition of $G_k(\x;\x')$, we have $\norm{F_k}_{L^2(\Omega)} \leq c_0^{-1} = c_1^{-1}k^{-1}$. The function $F_k$ also satisfies the equation
\begin{align*}
(\Delta + k^2V)F_k = -\Pi_kg \text{ in } \Omega, \qquad \pa_{\nu}F_k = 0 \text{ on } \pa\Omega.
\end{align*}
In this case, we use $\Pi_k$ to denote the projection operator away from the eigenfunctions $w_m$, with $|\la_m-k^2|\leq c_0$. This means that $\norm{\nabla F_k}_{L^2(\Omega)} \leq Ckc_0^{-1}=Cc_1^{-1}$. {In a Lipschitz domain, the Sobolev trace estimate states that the $L^2(\pa\Omega)$ norm of a function is controlled by its $H^{1/2+\delta/2}(\Omega)$-norm for any $\delta>0$.} Therefore, returning to the expression in \eqref{eqn:case2d}, for any $\delta>0$, there therefore exists $C_{\delta}>0$ such that
\begin{align*}
\int_{\Omega}\tilde{u}(\x)g(\x) \ud \x = \int_{\pa\Omega}F_k(\x')\partial_{\nu}u(\x') \ud s(\x')
\end{align*}
can be bounded in absolute value by
\begin{align}
\label{eqn:trace1}
\norm{F_k}_{L^2(\pa\Omega)}\norm{\partial_{\nu}u}_{L^2(\pa\Omega)} \leq C_{\delta}\norm{F_k}_{H^{1/2+\delta/2}(\Omega)}\norm{\partial_{\nu}u}_{L^2(\pa\Omega)}\leq  C_{\delta}c_1^{-1}k^{-1/2+\delta/2}\norm{\pa_\nu u}_{L^2(A)}^{1/2}.
\end{align}
Here we have again also used the estimate on $\partial_{\nu}u$ from \eqref{eqn:g1}, and Sobolev interpolation to bound $\norm{F_k}_{H^{1/2+\delta/2}(\Omega)}$. This gives the bound 
\begin{align*}
\norm{\tilde{u}}_{L^2(\Omega)} \leq C_{\delta}c_1^{-1}k^{-1/2+\delta/2}\norm{\pa_\nu u}_{L^2(A)}^{1/2}.
\end{align*}
Since $\tilde{u}$ and $h$ are orthogonal in $L^2(\Omega)$, {we can integrate the equation in \eqref{eqn:case2c} against $\tilde{u}$ and integrate by parts to obtain} 
\begin{align*}
\norm{\nabla\tilde{u}}_{L^2(\Omega)} \leq C_{\delta}c_1^{-1}k^{-1/2+\delta/2}k\norm{\pa_\nu u}_{L^2(A)}^{1/2} + \left|\int_{\pa\Omega}\tilde{u}\partial_{\nu}u\ud s \right|.
\end{align*} 
Thus, $\norm{\nabla\tilde{u}}_{L^2(\Omega)}  \leq C_{\delta}c_1^{-1}\left(1+k^{1/2+\delta/2}\right)\norm{\pa_\nu u}_{L^2(A)}^{1/2}$, and {using the Sobolev trace estimate again gives}
\begin{align} \label{eqn:case2e}
\norm{\tilde{u}}_{L^2(\pa\Omega)} \leq C_{\delta}\norm{\tilde{u}}_{H^{1/2+\delta/2}( \Omega)} \leq C_{\delta}c_1^{-1}\left(1+k^{\delta}\right)\norm{\pa_\nu u}_{L^2(A)}^{1/2},
\end{align}
{for some possibly modified $C_\delta>0$.} 
Therefore, for $c_1$ sufficiently small, depending on $c^*$, we can use \eqref{eqn:case2b} to see that, {since $\tilde u=u-w$,}
\begin{align*}
c\norm{w}_{L^2(\Omega)} - \norm{u}_{L^2(\paOA)} \leq {\|\tilde u\|_{L^2(\paOA)}}\leq C_{\delta}c_1^{-1}\left(1+k^{\delta}\right)\norm{\pa_\nu u}_{L^2(A)}^{1/2}.
\end{align*}
Thus, the estimate on $u$ from \eqref{eqn:g1} implies that $\norm{w}_{L^2(\Omega)} \leq C_{\delta}\left(k^{-1}+(1 +k^{\delta})c_1^{-1}\right)  \norm{\pa_\nu u}_{L^2(A)}^{1/2}$.
{Then, integrating the equation in \eqref{eqn:case2a} against $w$ and integrating by parts,} we obtain $$\norm{w}_{H^1(\Omega)} \leq C_{\delta}(1+k)^{1+\delta}c_1^{-1}\norm{\pa_\nu u}_{L^2(A)}^{1/2}.$$
{We use the Sobolev trace estimate a final time to get
\begin{align}
\label{eqn:trace3}
    \norm{w}_{L^2(\pa\Omega)} 
    \leq C_{\delta}\norm{w}_{H^{1/2+\delta/2}(\Omega)} \leq C_{\delta}(1+k)^{1/2+\delta}c_1^{-1}\norm{\pa_\nu u}_{L^2(A)}^{1/2}.
\end{align}}
{Combining this with} \eqref{eqn:case2e} therefore implies that
\begin{align} \label{eqn:case2f}
\norm{ku}_{L^2(\pa\Omega)} \leq {\norm{k \tilde u}_{L^2(\pa\Omega)}+\norm{kw}_{L^2(\pa\Omega)}} \leq C_{\delta}(1+k)^{3(1+\delta)/2}c_1^{-1}\norm{\pa_\nu u}_{L^2(A)}^{1/2},
\end{align}
which completes the proof of the estimate in \eqref{eqn:Main1a}.

\qed
\begin{rem} \label{rem:sharpness}
From the proof above, we see that any improvement in the dependence on $k$ in the upper bounds in \eqref{eqn:trace1}, \eqref{eqn:case2e}, and \eqref{eqn:trace3} coming from an application of the trace theorem, or an improvement in the lower bound in Lemma \ref{lem:control}, will lead to a sharper bound in Theorem \ref{thm:Main-W}. Note that from \cite{tataru-trace} the sharp estimate for the $L^2(\pa\Omega)$ Dirichlet trace of a $L^2(\Omega)$-normalized Neumann eigenfunction of frequency $k^2$ is $O((1+k)^{1/3})$ (with this bound attained for a sequence of eigenfunctions on the disc), compared to the bound of $O((1+k)^{1/2+\delta})$ that one obtains from a direct application of the Sobolev trace estimate.
\end{rem}

\section{(Non-sharp) Estimates for Obstacles using our machinery}
\label{sec:obs}

Consider the following boundary value problem on $\Omega$ described in Figure \ref{fig:obstacle}
\[
\Delta u + k^2 u = 0,\qquad  \ u|_{\partial \Omega_5} = 0,\qquad \ \pa_{\nu}u+iku|_{\partial \Omega_j} = f_j \ \text{for} \ j = 1,2,3,4.
\]  
We will demonstrate that comparable results to those in Theorem \ref{thm:Main} hold in this case, though we will restrict ourselves to the case of $k > c^* > 0$ uniformly bounded below away from $0$ and we will have less explicit constants arising in the estimates for large $k$. 
We prove for{{ any $0<\gamma < \tfrac{1}{2}$}} 
the slightly modified bounds
\begin{align}
\norm{(I-R)f}_{{H}^{1+\gamma}(A)} & \geq c_- (k) \norm{f}_{H^\gamma (A)}, \label{I-Ralt} \\
\norm{(I+R)f}_{{{ H^{\gamma}(A)}}} & \geq  c_+ (k) \norm{f}_{{{ H^{\gamma}(A)}}},   \label{I+Ralt}
\end{align}
where $c_{\pm} (k)$ {are constants depending only on $k$. Here the operator $R$ is defined analogously to before, now with Dirichlet boundary conditions on the obstacle $\pa\Omega_5 = \partial K$ for $\Omega_5$ an open, convex set contained in $\Omega$, and $A=\pa\Omega_4$ is the right side of the square.}
The proof of these bounds involve a boundary control estimate using microlocal analysis. With more careful microlocal methods, these bounds can likely be improved and optimized, though we do not pursue this here. {As we are not trying to track the dependence of the estimates on $k$, we now use the standard fractional Sobolev spaces $H^{\gamma}(A)$ on the boundary.}

\begin{figure}[h]
		\includegraphics[width=6cm]{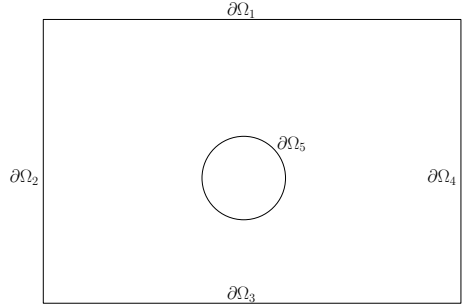} 
	\caption{A schematic of the obstacle problem.}
	\label{fig:obstacle}
\end{figure}

First, we consider estimate \eqref{I-Ralt}.  Following the first steps of the proof in Section \ref{sec:proof}, we may assume that $\| f \|_{H^\gamma (A)} = 1$, and so
\begin{align*}
   1   \leq \| k u \|_{H^\gamma (A)} + \| \partial_\nu u \|_{H^\gamma (A)} \leq
    k\| u \|_{H^{1+\gamma} (A)} +  \| \partial_\nu u \|_{H^\gamma (A)}.
\end{align*}
Therefore, the desired estimates on $I\pm R$ will follow once we prove the existence of $C_\pm(k)>0$ such that
{\begin{align} \label{obsbd}
 \| \partial_\nu u \|_{H^\gamma (A)} \leq C_-(k) \| u \|^{\frac14}_{{H}^{1+\gamma}(A)}, \qquad  \| ku \|_{H^\gamma (A)} \leq C_+(k) \| \pa_{\nu}u \|^{\frac14}_{{H}^{\gamma}(A)}.
\end{align}}

{We now proceed to replicate the rest of the proof from Section \ref{sec:proof}. Under the restriction to $k>c^*$ and since the obstacle is star-shaped, Theorem A.6 in \cite{graham2019helmholtz} ensures that the analogous elliptic estimates to Lemma \ref{lem:elliptic} continue to hold.} Lemma \ref{lem:control} also follows using the same proof since we have Dirichlet boundary conditions on the obstacle. We emphasize however that the bounds on the normal derivatives only hold on $\Omega_{1,2,3}:=\partial \Omega_1 \cup \partial \Omega_2 \cup \partial \Omega_3$, as the proof gives bounds only on the normal derivatives of the components of the boundary with impedance boundary conditions. Even with these lemmas in hand, the proof of Proposition \ref{prop:weak} involves integration by parts estimates with well chosen vector fields to fit the polygonal nature of the problem. In particular, the proof of Proposition \ref{prop:weak} would now pick up an error term involving the normal derivative of $u$ on the obstacle. As a result, we cannot in general use these techniques when there is an obstacle, and a modification is required.

{Therefore, we will establish both estimates in \eqref{obsbd} using the strategy of the proof of Proposition \ref{prop:weak2}. Note that we have restricted here to $k$ uniformly bounded away from $0$, and hence, we only need to consider Case $2$ of the proof.} The other key difference, needed 
to allow for obstacles in our estimates, is that we need to generalize the results in Lemma \ref{prop:efnN} to give boundary control on the three quadrilateral walls  $\partial \Omega_{1,2,3}$  away from the shared wall $A = \partial \Omega_{4}$ and the obstacle itself $\partial \Omega_{5}$.  
The resulting boundary control estimates that we require are of the following form.  We similarly write $\partial\Omega_{1,2,3,4}:=\partial \Omega_1 \cup \partial \Omega_2 \cup \partial \Omega_3\cup \partial \Omega_4$ to refer to the full set of boundaries for the quadrilateral.

\begin{prop} 
\label{prop:efnNobs}
Let $c^*>0$ and
{$w$ be an $L^2(\Omega)$-normalized function satisfying
\begin{align}
\label{app:eqns}
(\Delta + k^2 V)w = h \emph{ in } \Omega, \qquad \pa_{\nu}w = 0 \emph{ on } \pa \Omega_{1,2,3}, \qquad\ w = 0  \emph{ on } \pa \Omega_{5}
\end{align}
for some $h\in H^1(\Omega)$ and $k^2 >c^*$. Suppose that either $w = 0$ on $\pa\Omega_4$, or $\pa_{\nu}w = 0$ on $\pa\Omega_4$. Then, for each $\eps>0$ there exists  $C>0$ (depending only on $c^*$ and $\eps$) such that
\begin{align}
\label{eq:obscontalt}
\| w \|_{H^1 (\Omega)} \leq C k^\epsilon \left( \| h \|_{H^1 (\Omega)}+ \| w  \|_{H^1 (\partial\Omega_{1,2,3})} \right).
\end{align}}
\end{prop}


 This estimate does not appear in the literature to our knowledge, but it is very similar to estimates derived in \cite{Burq_Thesis,burq2004geometric} related to observability/control theory for Schr\"odinger equations in various geometries.  The key components required to prove such estimates are dynamical control of geodesics that intersect the control region and a resolvent estimate that determines how much one can concentrate on a non-degenerate hyperbolic orbit.  The proof of this estimate involves some technical microlocal analysis tools that exceed the scope of this section, but we give an overview of the proof in Appendix~\ref{app:obscont}.

{We now demonstrate how to obtain the first estimate in \eqref{obsbd} by adapting the proof of Proposition \ref{prop:weak2}}. 
Using Proposition \ref{prop:efnNobs} instead of Lemma \ref{prop:efnN}, we can approach the proof of this estimate using a Green's function very similarly to the proof of Proposition \ref{prop:weak2}. {The main difference is due to the gradient of $w$ on $\pa\Omega_{1,2,3}$ appearing in the right hand side of the estimate in Proposition \ref{prop:efnNobs}. This means that we require control on a higher order derivative of $u$ on $\pa\Omega_{1,2,3}$, and this is why our estimates on $I\pm R$ require fractional Sobolev spaces.}  

We now begin our adaptation of the proof of Proposition \ref{prop:weak2}. Namely, let
us take $\{\mu_m\}_{m\geq0}$ to be the eigenvalues of the problem
$$
\begin{cases}
(\Delta + \mu_m )w_m  = 0 & \text{ in } \Omega, \\
\pa_{\nu}w_m = 0 & \text{ on } \pa\Omega_{1,2,3}, \\
w_m = 0 & \text{ on } \pa\Omega_{4,5}, \\
\end{cases}
$$
with corresponding orthonormal eigenfunctions $w_m$, where $\partial\Omega_{4,5}:=\partial \Omega_4 \cup \partial \Omega_5$.  We note that we take Dirichlet boundary conditions on $A = \pa \Omega_4$ since we wish to observe control on the normal derivatives there. 


To prove the first estimate in \eqref{obsbd}, we may assume without loss of generality that 
\[
\| u \|_{H^{1+\gamma}(A)} < \| u \|_{H^{1+\gamma}(A)}^{\frac14}.
\]
By the analogous statement to Lemma \ref{lem:control}, it therefore follows that
\[
\| \partial_\nu u \|_{L^2 (\partial\Omega_{1,2,3})} + \| k u \|_{L^2 (\partial \Omega_{1,2,3})} \leq C(k)\| u \|_{H^{1+\gamma} (A)}^{\frac12},
\]
where here and throughout $C(k)$ is a constant depending only on $k$ (and which may change from line-to-line). Since $\pa_{\nu}u = -iku$ on $\pa\Omega_{1,2,3}$ and $\nabla u$ is bounded on $\pa\Omega_{1,2,3}$, by interpolation we have
\begin{align} \label{eqn:extra-control}
\| \nabla u \|_{H^\gamma (\partial\Omega_{1,2,3})} + \| k u \|_{L^2 (\partial \Omega)} \leq C(k)\| u \|_{H^1 (A)}^{\frac14}.
\end{align}
To now prove the first estimate in \eqref{obsbd}, let us take
\[
w = \sum_{|\mu_m - k^2| < c_0 (k)} (w_m,u) w_m,
\] 
and note then that 
\begin{align} \label{eqn:w-equation}
(\Delta + k^2) w = h, \ \ h = \sum_{|\mu_m -k^2| \leq c_0} (k^2 - \mu_m) (w_m,u) w_m
\end{align}
and
\begin{align} \label{eqn:h-obstacle}
\| h \|_{H^1 (\Omega)}^2 \leq C k^2 c_0^2  \sum_{|\mu_m -k^2| \leq c_0} | (w_m,u) |^2 = C k^2 c_0^2\norm{w}^2_{L^2(\Omega)}.
\end{align}
{Here $c_0 = c_0(k)$ is a constant depending on $k$ to be prescribed below.}  Defining the modified Green's function
\[
{\tilde G}_{k,c_0} (\x,\x') =  \sum_{|\mu_m -k^2| > c_0} \frac{1}{k^2-\mu_m^2} w_m (\x) w_m (\x'),
\]
we see that for $|\mu_m -k^2| \leq c_0$, we have
\[
\int_\Omega w_m (\x') {\tilde G}_{k,c_0} (\x,\x') d\x' = 0.
\]
Letting ${\tilde u} = u-w$ gives
$$
\begin{cases}
(\Delta + k^2 ){\tilde u}  = -h & \text{ in } \Omega \\
\pa_{\nu} {\tilde u} = \pa_{\nu} u & \text{ on } \pa\Omega_{1,2,3}, \\
{\tilde u} = 0 & \text{ on } \pa\Omega_{5}, \\
{\tilde u} = u & \text{ on } A = \pa\Omega_{4}.
\end{cases}
$$
Thus,
\[
{\tilde u}(\x)  = - \int_{\pa \Omega_{1,2,3}} {\tilde G}_{k,c_0} (\x,\x') \pa_{\nu} u(\x') d\x' -  \int_{A} \pa_{\nu}{\tilde G}_{k,c_0}  (\x,\x')  u(\x') d\x'
\]
since $(\tilde u, w_m) = 0$ for $|\mu_m - k^2| \leq c_0$.

Following the proof of Proposition \ref{prop:weak2} mutatis mutandis, taking $\| g\|_{L^2 ( \Omega)} =1$, we construct the function 
$$F_k(\x) = -\int_\Omega {\tilde G}_{k,c_0} (\x,\x') g(\x')d\x'$$ 
satisfying the equation
\begin{align*}
(\Delta + k^2)F_k = -\tilde{\Pi}_k g \text{ in } \Omega, \qquad \pa_{\nu}F_k = 0 \text{ on } \pa\Omega_{1,2,3}, \qquad F_k = 0 \text{ on } \pa\Omega_{4,5}.
\end{align*}
Once again, we use $\tilde{\Pi}_k$ to denote the projection operator away from the eigenfunctions $w_m$, with $|\mu_m-k^2|\leq c_0$. {Using elliptic estimates in the square \cite{grisvard2011elliptic}, this ensures that $\norm{F_k}_{H^2(\Omega)} \leq C(k)$.} 
Combining this estimate with the expression
\[
\int_{\Omega} {\tilde u}(\x) g(\x) d\x= \int_{\pa \Omega_{1,2,3} } F_k (\x') \pa_{\nu} u (\x') d\x' + \int_{A} \pa_\nu F_k (\x') u (\x')d\x',
\]
{the estimates in \eqref{eqn:extra-control} thus give the $k$-dependent bound $\norm{\tilde{u}}_{L^2(\Omega)} \leq C(k)\norm{u}^{1/4}_{H^{1+\gamma}(A)}$.} 
Using the equation for ${\tilde u}$, and applying the elliptic estimate 
\begin{equation}
    \label{gris_ellest}
\| {\tilde u} \|_{H^{\frac32 + \gamma} (\Omega)} \leq C(k) \Big( \norm{\tilde{u}}_{L^2(\Omega)}+\| h \|_{H^{-\frac12+\gamma} (\Omega) } + \| \tilde u \|_{H^{1 + \gamma} (A)} +   \| \partial_\nu \tilde u \|_{H^{\gamma} (\partial \Omega_{1,2,3})}\Big),
\end{equation}
{we also have the bound
\begin{align} \label{gris_ellest2}
    \| {\tilde u} \|_{H^{\frac32 + \gamma}(\Omega)} \leq C(k)\Big(\| h \|_{H^{-\frac12+\gamma} (\Omega) }+ \norm{u}^{1/4}_{H^{1+\gamma}(A)} \Big).
\end{align}
The estimate \eqref{gris_ellest} is stated in \cite{grisvard2011elliptic}, Section 1.5 (specifically following the discussion of the proof Theorem 1.5.2.4) and is proved using interpolation estimates established in \cite{Grisvard1966}.
Recalling that $\tilde{u} = u-w$, and applying the trace estimate for convex polygons, this implies that
\begin{align*}
    \norm{w}_{H^{1+\gamma}(\pa\Omega_{1,2,3})} -  \norm{u}_{H^{1+\gamma}(\pa\Omega_{1,2,3})} \leq C(k)[\| h \|_{H^{-\frac12+\gamma} (\Omega) }+ \norm{u}^{1/4}_{H^{1+\gamma}(A)} ].
\end{align*}
Using \eqref{eqn:h-obstacle}, we can choose $c_0 = c_0(k)>0$ sufficiently small so that the factor of $Ck^{\eps}\norm{h}_{H^1(\Omega)}$ in the right hand side of \eqref{eq:obscontalt} can be incorporated in the left hand side of the inequality. Therefore, applying Proposition \ref{prop:efnNobs} for this choice of $c_0$,
we have
\begin{align*}
     \norm{w}_{H^1(\Omega)}  & \leq Ck^{\eps} \norm{w}_{H^1(\pa\Omega_{1,2,3})} \\
     & \leq C(k)[\| h \|_{H^{-\frac12+\gamma} (\Omega) }+ \norm{u}^{1/4}_{H^{1+\gamma}(A)} +\norm{u}_{H^{1+\gamma}(\pa\Omega_{1,2,3})} ].
\end{align*}
Again choosing $c_0 = c_0(k)>0$ sufficiently small to absorb $\norm{h }_{H^{-\frac12+\gamma}}$ into the left hand side, and using \eqref{eqn:extra-control}, we obtain
\begin{align} \label{eqn:w-estimate1}
    \norm{w}_{H^1(\Omega)}\leq C(k)\norm{u}^{1/4}_{H^{1+\gamma}(A)}.
\end{align}
The function $w$ satisfies mixed Dirichlet-Neumann boundary conditions and the elliptic equation in \eqref{eqn:w-equation}. From this elliptic equation satisfied by $w$, we can therefore convert the estimate in \eqref{eqn:w-estimate1} into an estimate on $w$ in $H^{3/2+\gamma}(\Omega)$, and inserting this in \eqref{gris_ellest2} implies that
\begin{align*}
    \norm{u}_{H^{3/2+\gamma}(\Omega)}\leq C(k)\norm{u}^{1/4}_{H^{1+\gamma}(A)}.
\end{align*}
Finally, the trace theorem implies the first estimate in \eqref{obsbd}.

The second estimate in \eqref{obsbd},
\begin{align*}
   \norm{ku}_{H^{\gamma}(A)} \leq C_{+}(k)\norm{\pa_{\nu}u}^{1/4}_{H^{\gamma}(A)}, 
\end{align*}
follows using the same strategy of proof both as above and in the proof of Proposition \ref{prop:weak2}, again using the estimate in Proposition \ref{prop:efnNobs} in place of Lemma \ref{prop:efnN}.


}

\section{Remarks and Examples} \label{sec:example}

\subsection{Impedance eigenfunctions on the square} \label{sec:example1}

In this subsection, we discuss the estimates in Theorems \ref{thm:injective}, \ref{thm:Main-W}, and \ref{thm:Main} both in terms of the spaces appearing and the dependence of the estimates on the frequency $k$. To do this, we will consider the special case where $\Omega$ is the unit square $[0,1]\times[0,1]$ and the potential $V\equiv 1$. We write down some explicit solutions to
\begin{align} \nonumber
(\Delta + k^2)u & = 0 \text{ in } S \\ \label{eqn:u-square}
\pa_{\nu} u + iku & = 0 \text{ on } \pa S\backslash A \\ \nonumber
\pa_{\nu} u + iku & = f \text{ on } A,
\end{align}
by first separating variables and looking at the eigenvalues and eigenfunctions of the impedance problem on $[0,1]$.
\begin{lem} \label{lem:1d}
The eigenvalues of 
\begin{align*}
    w''(y) & = -\lambda^2w(y)  \text{ in } [0,1] \\
    w'(1) + ikw(1) & = -w'(0) + ikw(0) = 0
\end{align*}
are given by the solutions $\la = \la_n$ of
\begin{align*}
    e^{2i\la} = \frac{(1-k\la^{-1})^2}{(1+k\la^{-1})^2},
\end{align*}
with corresponding eigenfunctions
\begin{align*}
    w_n(y) = A_n\left\{(\la_n+k)e^{i\la_ny} + (\la_n-k)e^{-i\la_ny}\right\}.
\end{align*}
\end{lem}
This lemma follows immediately from solving the ODE on $[0,1]$ with the impedance boundary conditions. The solution of \eqref{eqn:u-square} with $f_n = w_n$ from Lemma \ref{lem:1d} is then given by $u_n(x,y) = v_n(x)w_n(y)$, with $v_n(x)$ satisfying $v_n''(x) = (\la_n^2 - k^2)v_n(x)$ on $[0,1]$ and the boundary conditions $-v_n'(0)+ikv_n(0)=0$, $v_n'(1)+ikv_n(1) = 1$. Solving this ODE, and setting $\mu_n^2 = \la_n^2 - k^2$ gives 
\begin{align} \label{eqn:vn-defn}
v_n(x) = \frac{(\mu_n+ik)e^{\mu_n x} + (\mu_n -ik)e^{-\mu_nx}}{(\mu_n+ik)^2e^{\mu_n} - (\mu_n-ik)^2e^{-\mu_n}}.
\end{align}
We use this construction to study the type of bounds that we can and cannot obtain on the operators $I\pm R$.
\begin{enumerate}
    \item[1)] We first show that for any fixed $k$, the operator $I-R$ is not bounded from below as an operator from $L^2(A)$ to itself: For $k$ fixed and $n$ large, the eigenvalues $\la_n$ from Lemma \ref{lem:1d} have the asymptotics
    \begin{align*}
        \la_n = n\pi + O_k(n^{-1}), \quad \mu_n = n\pi + O_k(1).
    \end{align*}
    Therefore, the functions $v_n(x)$ from \eqref{eqn:vn-defn} satisfy
    \begin{align*}
        v_n(1) = \tfrac{1}{n\pi} + O_k(n^{-2}), \quad v_n'(1) = 1 + O_k(n^{-1}).
    \end{align*}
    As a result of these asymptotics, since $(I-R)f_n(y) = 2ku_n(1,y) = 2kv_n(1)w_n(y)$, we have
    \begin{align*}
        \frac{\norm{(I-R)f_n}_{L^2(A)}}{\norm{f_n}_{L^2(A)}} = \frac{2k}{n\pi} + O_k(n^{-2}).
    \end{align*}
    Since $n$ can be any positive integer, we indeed see that $I-R$ is not bounded from below as an operator from $L^2(A)$ to itself. 
    
    \item[2)] We next use the functions $u_n(x,y)$ to show that the constant appearing in the lower bound on the operator $I+R$ in Theorem \ref{thm:Main} cannot be taken independent of $k$, for large $k$. Since 
    \begin{align} \label{eqn:vn2}
        (I+R)f_n(y) = 2\pa_xu_n(1,y) = 2v_n'(1)w_n(y), \qquad f_n(y) = w_n(y),
    \end{align}
    we will do this by choosing $n$ (depending only on $k$) to make $v_n'(1)$ small. By the definition of $v_n$ from \eqref{eqn:vn-defn}, we therefore want to make $\mu_n$ as small as possible. Given a small $\alpha>0$, we consider the sequence of $k = k_n$ such that $k+k^{\alpha}$ is equal to $n\pi$, and write 
    \begin{align*}
        \la_n = k + k^{\alpha}+i\delta_n.
    \end{align*}
    Then, $\la_n$ is an eigenvalue provided $\delta_n$ satisfies
    \begin{align*}
        e^{2ik + 2i{k}^{\alpha} - 2\delta_n} = \frac{({k}^{\alpha} +i\delta_n)^2}{(2k +{k}^{\alpha}+i\delta_n)^2}.
    \end{align*}
    For large $k$, we therefore require
    \begin{align*}
        e^{-2\delta_n} = \frac{k^{2\alpha}}{(2k+k^{\alpha})^2} + O(\delta_n k^{-2}),
    \end{align*} 
    which has a solution for $\delta_n$ with $\delta_n = O(\log (k))$. In particular, $\mu_n^2 = \la_n^2 - k^2 = 2k^{1+\alpha} + O(k^{2\alpha})$, and so from \eqref{eqn:vn-defn},
    \begin{align*}
        |v_n'(1)| = O(\mu_n k^{-1}) = O(k^{-1/2+\alpha/2}).
    \end{align*}
    Using this in \eqref{eqn:vn2} shows that 
    \begin{align*}
        \frac{\norm{(I+R)f_n}_{L^2(A)}}{\norm{f_n}_{L^2(A)}} \leq Ck^{-1/2+\alpha/2},
    \end{align*}
    and hence the constant in the second estimate in Theorem \ref{thm:Main} must tend to $0$ at least at this rate as $k$ increases.
    \end{enumerate}

\subsection{Use of $W$ in the numerical scheme} \label{sec:example2}

As we described in the Introduction, our motivation for analysing the operator $W$ is in its use in merge operations in numerical schemes. In the scheme used in \cite{gillman2015spectrally}, the operator $W$ is only applied to operators $Q_j$ that take the following form:

Given $h\in L^2(\paSA)$, let $v_j$ solve
$$
\begin{cases}
    (\Delta + k^2V_j)w_j & = 0 \text{ in } S\\
    \pa_{\nu}w_j + ikw_j & = h \text{ on } \paSA \\
    \pa_{\nu}w_j + ikw_j & = 0 \text{ on } A.
\end{cases}
$$
We again then define the operator $Q_j$ on $L^2(A)$ by $Q_jh = \pa_{\nu}w_j - ikw_j\big|_{A}$. Then, under Assumption \ref{assum:potential-main} on the potential $V_j$, using Proposition \ref{prop:elliptic1}, the functions $w_j$ satisfy the following elliptic estimates,
\begin{align*}
    \norm{kw_j}_{L^2(A)} + \norm{\pa_{y}w_j}_{L^2(A)} \leq C\norm{h}_{L^2(\paSA)}.
\end{align*}
In particular, $\norm{Q_jh}_{L^2(A)} \leq C\norm{h}_{L^2(\paSA)}$. Moreover, since $\pa_{x}w_j + ikw_j = 0$ on $A$, we have
\begin{align*}
    k^{-1}\pa_yQ_jh = k^{-1}\pa_y\pa_xw_j - i\pa_yw_j\big|_{A} =  -2i\pa_yw_j\big|_{A}.
\end{align*}
Therefore, using Theorem \ref{thm:Main-W} we obtain
\begin{align*}
    \norm{WQ_jh}_{L^2(A)} &\leq C^*_{\delta}(1+k)^{3(1+\delta)}\left(\norm{Q_jh}_{L^2(A)} + \norm{k^{-1}\pa_yQ_jh}_{L^2(A)}\right) \\
    & \leq C_\delta^*(1+k)^{3(1+\delta)}\norm{h}_{L^2(\paSA)}.
\end{align*}
In particular, for this composition of operators we do not have the loss of a derivative in the estimate. Moreover, this bound is uniform in $k$, for small frequency $k$. Since, from Lemma \ref{lem:R-boundedness}, $R_j$ is uniformly bounded as an operator from $L^2(A)$ to itself, the above estimates for $WQ_j$ also hold for the operators $WR_mQ_j$. This is important because in the merge process from  \cite{gillman2015spectrally}, as shown in the expressions in \eqref{merge-eqn2} and \eqref{merge-eqn3}, it is these compositions of operators that are used iteratively to reconstruct the ItI operators in the original domain in terms of those of the hierarchical tree of square or rectangular boxes in the partition.

\appendix

\section{The Obstacle Control Theory Estimate}
\label{app:obscont}

In this appendix we provide an overview of the microlocal analysis results and tools required to establish the boundary control estimate in Proposition \ref{prop:efnNobs}. In particular, we recall here the 
{\it black box} control theory and observability machinery from the work of Burq-Zworski \cite{burq2004geometric}, which in turn relies on resolvent estimates for exterior scattering problems.  Resolvent estimates for various scattering settings can be read about, for example, in  \cite{dzbook}, though we will recall the necessary literature for our setting below.    Hence, we especially refer the reader to Section $6$ of \cite{burq2004geometric} for several applications that begin with resolvent estimates of the form considered here.  

The results from \cite{burq2004geometric} are designed to apply to a wide variety of situations and the proofs involve microlocal analytic techniques such as propagation of singularities estimates and semiclassical defect measures that are beyond the scope of the current study.  Hence, we will state the estimates that can be inferred in our setting of a domain with mixed boundary conditions and an obstacle with Dirichlet boundary conditions, but with notation simplified to that we have used above in our analysis.  While the estimate \eqref{eq:obscontalt} in Proposition \ref{prop:efnNobs} does not appear in the literature to our knowledge, it is a relatively straightforward, albeit technically challenging, application of the black box machinery.  Hence, we forego some details here for clarity.  

As in Section \ref{sec:obs} (see Figure \ref{fig:obstacle}), we will consider the operator $-\Delta$ on the rectangle $\Omega$.  Here, we are imposing Dirichlet boundary conditions on the obstacle boundary, referred to as $\partial \Omega_5$, Dirichlet or Neumann boundary conditions on the gluing side, referred to as $\partial \Omega_4$, and imposing Neumann boundary conditions on the remaining exterior boundary components $\partial \Omega_{1,2,3}$.  In addition,  we will consider the domain, $X$, which denotes a properly selected double of $\Omega$ that is a compact manifold having reflected across $\partial \Omega_4$, and $Y$, an open subset of this manifold $X$ containing the resulting obstacles.  See Figure \ref{fig:ref1} for a sketch of what we intend with $X$ and $Y$.  For  simplicity, we will still refer to the outer boundary of the manifold $X$ as $\partial \Omega_{1,2,3}$.

The proofs of the main Theorems $1$ and $8$ in \cite{burq2004geometric} are given in Section $6$ of that work and are related to observability of time-dependent equations, though they are closely related to our desired estimates.  Indeed, the observability estimates in these theorems are seen to be equivalent to a family of resolvent estimates that are related to underlying geometry of the domains in question. For solutions $w$ of the equations in \eqref{app:eqns}, the resolvent estimates in \cite{burq2004geometric} are very similar to the estimates given in \eqref{eq:obscontalt}. The main results that we require to obtain such an estimate are those found in Theorem $6$ of \cite{burq2004geometric}, which allow us to establish the following estimate,
\begin{equation}
\label{bz:sc4}
\| w \|_{H^1 (\Omega)} \leq  \log (1+k^2)\left(  \langle k^2 \rangle^{-\frac12} \| (-\Delta -  k^2) w \|_{H^1 (\Omega)}+ \| w \|_{H^1 (\partial \Omega_{1,2,3})}   \right)
\end{equation}
for our specific geometry. Proving this estimate will clearly result in the proof of Proposition \ref{prop:efnNobs}. The logarithm in the above estimates will be related to a natural logarithmic loss in a resolvent estimate in the exterior to convex obstacles (see \eqref{eqn:obstacle-resolvent} below). 

In order to prove \eqref{bz:sc4}, one must establish first that elliptic estimates hold on the region away from $Y$ corresponding to boundary control results of the form
\begin{equation}
\label{bz:sc1}
\| \chi_{_{\!X \setminus Y}} w \|_{H^1 (\Omega)} \leq C \left( \langle  k^2 \rangle^{-\frac12} \| (-\Delta - k^2) w \|_{H^1(\Omega)} + \| w \|_{H^1 (\partial \Omega_{1,2,3})}  + \mathcal{O} (\langle  k^2 \rangle^{-\infty} ) \| w \|_{H^1 (\Omega)} \right),
\end{equation}
where $\chi_{_{\!X \setminus Y}}$ is a smooth cut-off function supported in a neighborhood $X \setminus Y$ and $\chi_{_{\!X \setminus Y}} = 1$ on $X \setminus Y$.
Second, we require a resolvent estimate on the {\it black box} exterior region given by taking the obstacles in $Y$ embedded in all of $\mathbb{R}^2$.  More precisely, take $\Omega_5'$ to be the union of the obstacle with boundary $\partial\Omega_5$ and its reflection about $A$, then take
\begin{align*}
(-\Delta - k^2 )v = h\; \text{ in } \mathbb{R}^2 \setminus \Omega_5', \qquad v = 0  \text{ on } \pa \Omega_{5}',
\end{align*}
where the operator $-\Delta - k^2$ is taken to have outgoing radiation conditions.  Then, we require that for a similarly defined cut-off function $\chi_{_{\! Y}} $ to a neighborhood of the black box region $Y$ that 
\begin{equation}
\label{bz:sc2}
\|  v \|_{H^1 (  \mathbb{R}^2 \setminus \Omega_5')} \leq \log (1+k^2)  \left(  \langle  k^2 \rangle^{-\frac12} \| \chi_{_{\!Y}}  (-\Delta - k^2) \chi_{_{\!Y}}  v \|_{H^1 ( \mathbb{R}^2 \setminus \Omega_5')}  \right) , \ \ |k| \to \infty,
\end{equation}
where $v$ is taken to be a smooth function, compactly supported within the support of $\chi_{_{\!Y}}$, and again the logarithmic dependence upon $k$ is related to the geometry of the black box region.

The estimate \eqref{eq:obscontalt} will therefore follow directly from Theorem $6$ of \cite{burq2004geometric} provided we can establish the necessary elliptic estimates on $X \setminus Y$ to establish \eqref{bz:sc1} and apply a resolvent estimate exterior to convex obstacles for \eqref{bz:sc2}.

Estimate \eqref{bz:sc1} follows from a series of operations involving reflection operations on the domain and applying the {\it Lifting Lemma} of Bardos-Lebeau-Rauch \cite[Theorem 2.2]{bardos1992sharp}.  We recall the statement of the result here for solutions of the wave equation.

\begin{lem}[\cite{bardos1992sharp}, Theorem 2.1]
Let $M$ be a compact Riemannian manifold with smooth boundary, $\partial M$.  Suppose that $q \in T^* (\partial M)$ is a non-diffractive point and that $u$ is a distribution defined in a sufficiently small neighborhood of $q$ in $M$, say $U$, such that
\begin{equation}
    \label{lift}
-\Delta u \in C^\infty (U), \ u|_{\partial M} \in H^s (U \cap \partial M), \ \nabla u |_{\partial M} \in H^{s-1} (U \cap \partial M).
\end{equation}
Then, $u \in H^s (U)$ with the proper interpretation of microlocal regularity up to the boundary.
\end{lem}

The result essentially states that anything near a smooth part of the boundary is controlled by the boundary.  Trivially, this extends to eigenfunctions.  Estimate  
\eqref{bz:sc1} and its classical analog would both follow from \eqref{lift} were the boundary of our quadrilateral domain smooth, as discussed in \cite{burq2004geometric}.  Since our boundary is piecewise smooth and in Proposition \ref{prop:efnNobs} we are studying mixed homogeneous Neumann and Dirichlet boundary conditions, we need a slight modification.  To handle our domain and boundary conditions, on a rectangle with Neumann (resp. Dirichlet) boundary conditions, we can do a series of even (resp. odd) reflections and explore boundary control in a domain with joint Neumann/periodic boundary conditions.  First, we perform an odd reflection about $\partial \Omega_4$ when it is a Dirichlet exterior boundary or an even reflection when it is Neumann such that we have a domain with two obstacles and care as such only about boundary control on all $4$ sides of the new domain, see Figure \ref{fig:ref1}. 

\begin{figure}[h]
		\includegraphics[width=7cm]{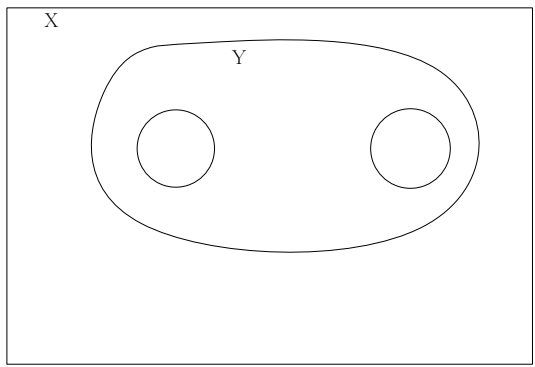} 
	\caption{A schematic of the obstacle problem after the reflection to remove the $\partial \Omega_4$ boundary.}
	\label{fig:ref1}
\end{figure}

In the region $X \setminus Y$, we thus have the boundary control in \eqref{bz:sc1} if the boundary were smooth.  In our case, we must simply do another even reflection first in the vertical direction, giving a cylindrical domain with Neumann boundary conditions on the vertical edges and periodic on the horizontal edges, see the Left panel of Figure \ref{fig:ref2} where we represent periodic boundary conditions with {\color{red} red} and Neumann with {\color{blue} blue} (color online).  This gives boundary control for all components of the function that will hit the now smooth Neumann boundaries.  A symmetric reflection gives a cylindrical domain with Neumann boundary conditions on the horizontal edges and periodic on the vertical edges, see the Right panel of Figure \ref{fig:ref2} where we represent periodic boundary conditions with {\color{red} red} and Neumann with {\color{blue} blue} (color online).  Applying this sequentially proves the boundary control from all remaining geodesics that would only intersect the horizontal boundary.  

\begin{figure}[h]
\begin{center}
		\includegraphics[width=5.5cm]{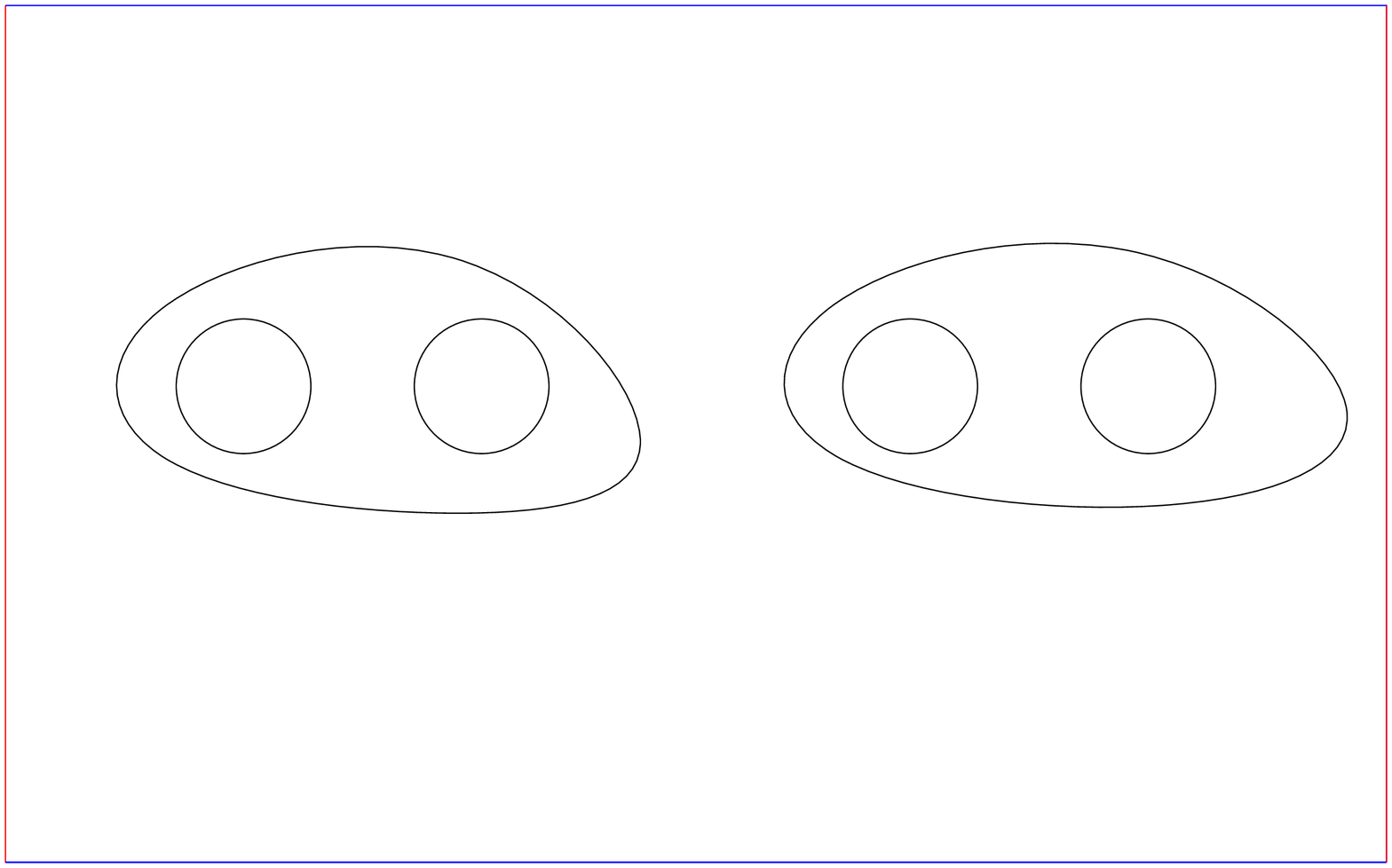} \quad \quad
		\includegraphics[width=6cm]{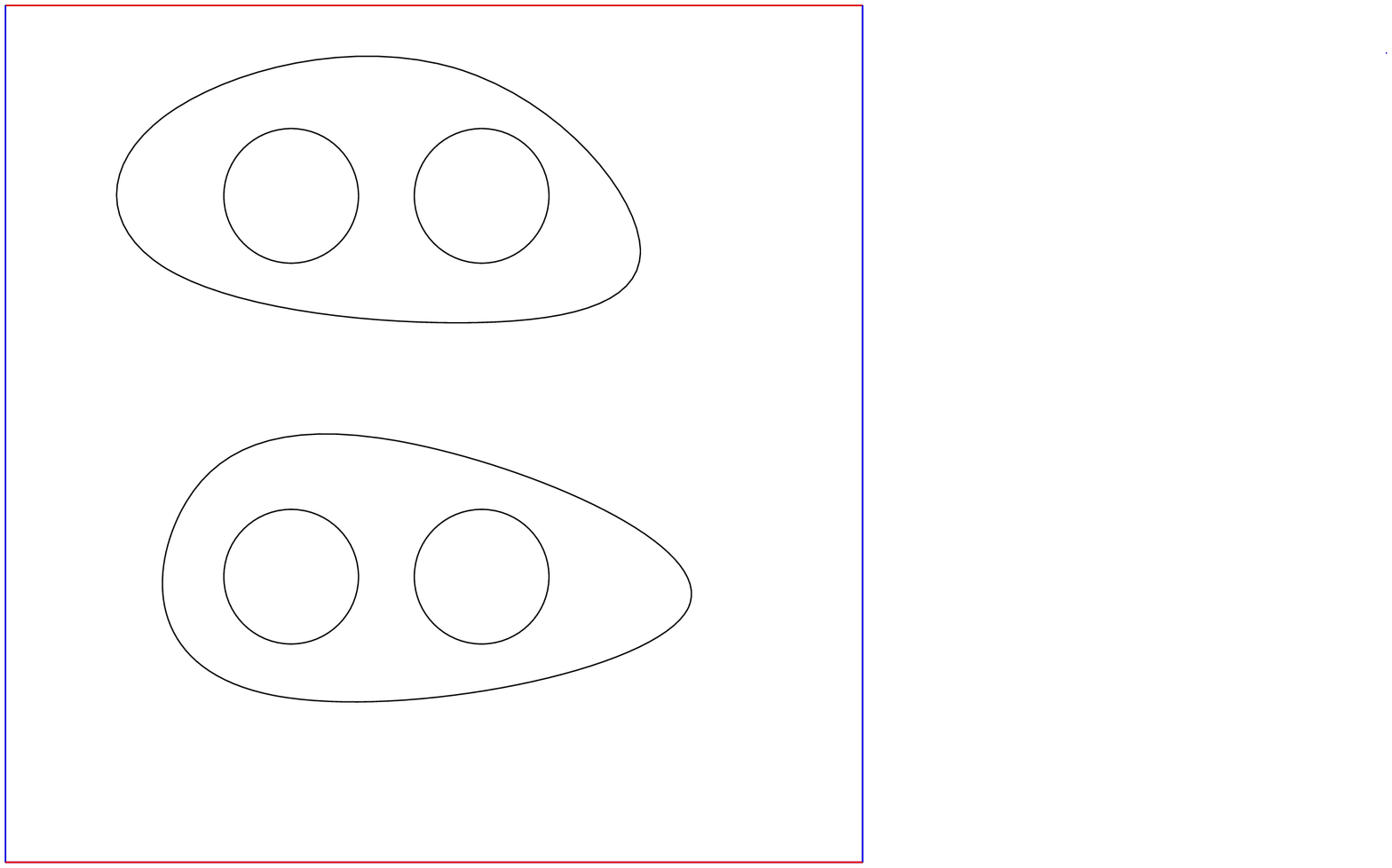}
	\caption{A schematic of the obstacle problem after horizontal reflection (Left) and vertical reflection (Right).}
	\label{fig:ref2}
	\end{center}
\end{figure}

The estimate in \eqref{bz:sc2} for the black box in a neighborhood of the obstacle and shared domain wall, follows from the work of Ikawa (\cite{ikawa1983poles}) and G\'erard (\cite{gerard1988asymptotique}).  This states that the resolvent $\mathcal{R} (k) = (-\Delta - k^2)^{-1}$ exterior to a collection of Dirichlet convex obstacles satisfies
\begin{align} \label{eqn:obstacle-resolvent}
\| \chi \mathcal{R} (k) \chi \|_{L^2 \to L^2} \leq C_0 \frac{ \log \langle k \rangle}{ \langle k \rangle}.
\end{align}
See also the recent book of Dyatlov-Zworski \cite{dzbook}, Chapter $6$ and specifically related estimates in Section $6.3$.  In the case of single obstacle scattering, one has
\[
\| \chi \mathcal{R} (k) \chi \|_{L^2 \to L^2} \leq C_0 \frac{ 1}{ \langle k \rangle}
\]
see \cite{tang2000resonance}.

Therefore, we have established the appropriate estimates \eqref{bz:sc1} and \eqref{bz:sc2}, and hence the proof of \eqref{eq:obscontalt} thus follows from Theorem $6$ of \cite{burq2004geometric}.

\bibliographystyle{alpha}
\bibliography{Refs}

\end{document}